\documentclass[11pt]{article}

\usepackage{amssymb, amsfonts, amsthm,amsmath,tikz,float,bm}
\usepackage{a4wide,color}

\newcommand{\PPo}[1]{\PP \left[\,#1\,\right]}
\newcommand{\PPol}[1]{\PP_{Z_{\tau}=\ell} \left[\,#1\,\right]}

\newcommand{\Ex}[1]{\mathbb{E} \left[\, #1\,\right]}

\newcommand{\Var}[1]{\mathrm{Var}\left[#1\right]}

\newcommand{\PP}{\mathbb{P}}

\newcommand{\numdeg}{R}
\newcommand{\ratdeg}{r}
\newcommand{\smalldeg}{Q}
\newcommand\xvec{{\boldsymbol{x}}}

\newtheorem{theorem}{Theorem}[section]

\newtheorem{lemma}[theorem]{Lemma}

\newtheorem{remark}[theorem]{Remark}

\title{Do random initial degrees suppress concentration in preferential attachment graphs?}

\author{T.~Makai$^{1}$, F.~Polito$^{2}$ and L.~Sacerdote$^{2}$ \\ \\
\footnotesize{1 -- Mathematisches Institut, Ludwig-Maximilians-Universität München, Germany} \\ \footnotesize{makai@math.lmu.de} \\ \footnotesize{2 -- Mathematics Department ``G.~Peano'', University of Torino, Italy} \\ \footnotesize{\{federico.polito, laura.sacerdote\}@unito.it}}

\begin{document}

    \maketitle

    \begin{abstract} 
        \noindent We consider the open problem concerning the possible lack of concentration of the degree distribution in preferential attachment graphs with random initial degree, when its distribution is characterized by extremely heavy tails of power-law type.
        We show that the addition of such a large number of edges causes a significant upset of the degree distribution, leading to its non-concentration. Furthermore, we show that
        the smallest value of the exponent for which the degree distribution exhibits concentration is 2.

        \medskip    

    \end{abstract}

    \section{Introduction}

        The so-called Barabási--Albert model \cite{MR2091634}
        describes
        a dynamical random graph in which the number of vertices always increases by one at each discrete time step. 
        Additionally, in each step a constant and deterministic number of edges are added which connect the newly born vertex to the existing ones. The attachment occurs with probabilities proportional to the degree of the vertices present in the graph. In \cite{MR1824277} the initial condition is specified and the overall model is formalized mathematically. A large number of variants as well as specialized studies of this model appeared in the literature during the last two decades. This corpus is so vast that it prevents any fair citation attempt. 
        We refer only to some of the most recent: \cite{MR4467842,MR3912097,MR3161480,MR3668381,MR4269210,MR4312838,MR4522354,MR3997484,MR4195181,MR4193887}.

        Different features of such models can be analyzed. One of them which is fundamental is the dynamics of the stochastic degree process. This process, say $\{\bm{R}_t\}_{t \in \mathbb{N}}$, evolves in discrete time and for each $t$ it is an infinite dimensional vector whose elements are the proportion of vertices of each given degree at time $t$. Note that by definition the sum of the elements of $\bm{R}_t$ equals unity for each given $t$. 
        A frequently considered problem in the literature is related to the asymptotic behaviour of this process. In particular, the interest focuses on the existence of $\bm{R}_\infty = \lim_{t \to \infty} \bm{R}_t$. If such limit exists  at least in probability and it is a real deterministic infinite vector (i.e.\ it degenerates), then it is called \textit{degree distribution}.  For the Barabási--Albert model and many of its variants (see e.g.\ \cite{MR1824277,MR4142215,MR2511283,MR3776186}) the degree distribution is a real sequence decaying as a power law with some characteristic exponent. Results relative to the above-mentioned form of convergence are known as concentration results.
        As far as the number of newly added edges is deterministic and constant the concentration results for $\bm{R}_t$ always hold true (see for instance \cite{MR3617364}, Chapter 8, and the references therein).
        However, when the initial degrees are random, it is not straightforward to establish whether the concentration of the process $\{\bm{R}_t\}_{t \in \mathbb{N}}$ occurs. In fact, this heavily depends on the distributional properties of the initial degrees themselves.
        A variation of the Barabási and Albert model \cite{MR2091634,MR1824277} admitting random initial degrees was introduced by Deijfen, van den Esker, van der Hofstadt and Hooghiemstra \cite{MR2480915} and it is known as the PARID-model (Preferential Attachment with Random Initial
        Degrees).

        The PARID-model depends on two quantities, $X$ a positive integer valued random variable and $\delta$ a fixed parameter which satisfy $\delta+\min\{x: \PPo{X = x}>0\}>0$.
        The construction of the model makes use of a sequence $\{X_i\}_{i\ge 1}$ of independent copies of $X$. Further, define $\Lambda(t)=\sum_{i=1}^t X_i$.
        In the first step of the process we add two vertices $v_0$ and $v_1$ and connect these two vertices with $X_1$ edges.
        In each subsequent step $t\ge 2$ of the process a new vertex, $v_t$ along with $X_t$ edges are added.
        Each of these $X_t$ edges is generated independently in the following manner: a vertex $v_i$, where $i\le t-1$, is selected with probability
        $$
            \frac{d_i(t-1)+\delta}{2\Lambda(t-1)+t\delta},
        $$
        where $d_i(t-1)$ denotes the degree of vertex $v_i$ after step $t-1$, and the edge $\{v_i,v_t\}$ is added to the multigraph. Note that in this way loops cannot be generated.
        In contrast to the preferential attachment model described in \cite{MR1824277} the edges added to the graph in the same step have no effect on the probability that a given vertex is selected.

        In this paper we consider the degree sequence of graphs created by the PARID-model in which for simplicity we set $\delta = 0$.
        More precisely, let $\numdeg_k(t)$ be the random number of vertices of degree $k$ and $\ratdeg_k(t)=\numdeg_{k}(t)/(t+1)$ be the proportion of vertices of degree $k$.
        Our aim is to establish if it exists a sequence $b_k$ such that with high probability\footnote{with probability tending to 1 as $t\to\infty$} (whp) we have $\ratdeg_k(t)=b_k+o(1)$ for every $k\ge 1$.  

        In the context of a more general model, when $X$ is bounded this question is answered affirmatevely by Cooper and Frieze \cite{MR1966545}. In addition, Deijfen, van den Esker, van der Hofstadt and Hooghiemstra \cite{MR2480915} consider the case in which $X$ has a finite moment of order $1+\varepsilon$ for some constant $\varepsilon > 0$. In particular the finite moment condition holds if $\PPo{X=k}=k^{-\alpha}\ell(k)$, $\alpha>2$, where $\ell(k)$ is a slowly-varying function (and hence $X$ is a regularly-varying random variable of parameter $\alpha$). In this case they show that the degree distribution is still regularly-varying of parameter
        $$
            \min\{\alpha,3+\delta/\Ex{X_1}\},
        $$
        and conjecture that this should hold when $\alpha\in [1,2]$ as well (Conjecture 1.4), although it is unclear how a regularly-varying random variable with parameter $\alpha=1$ would be defined. When concentration of the empirical degree sequence can be proved, the typical choice for $b_k$ is $\lim_{t\to\infty}\Ex{\ratdeg_k(t)}$. Bhamidi \cite{bhamidi2007universal} calculates this quantity (Theorem 40), when $\alpha\in(1,2]$, but does not prove the required concentration result, leaving the question open. Here we address this problem. Interestingly, when $X$ is regularly-varying of parameter $\alpha \in (1,2)$ 
        no concentration occurs (Theorem \ref{thm:main}) the conjecture lose its significance and should be at least reformulated.
        On the contrary, if $\alpha = 2$, concentration occurs (Theorem \ref{thm:main2}) and the conjecture holds.

        The paper is organized as follows. Section \ref{results} contains our main results described by the above-mentioned theorems with some related comments, while in Section \ref{proofs} we collect their proofs together with the necessary lemmas.

    \section{Main results}\label{results}

        We first describe some useful notations which we will use extensively throughout the paper.
        In the following we will frequently need to evaluate the asymptotic behaviour of discrete-time stochastic processes. Denoting by $(H_t)_{t=1}^\infty$ a real-valued discrete-time stochastic process, the following notation will be used:
        \begin{align}\label{label}
            H_t = f(t) + o(g(t)).
        \end{align}
        Note that with this writing we put the randomness in $o(g(t))$ while $f(t)$ accounts for the deterministic temporal asymptotic behaviour. More specifically the convergence of the stochastic process $M_t = o(g(t))$ can be twofold. It can occur with high probability (that is, in probability) if
        $$
            \forall \, k>0 \quad \lim_{t \to \infty}\mathbb{P}(|M_t|<k g(t)) =1,
        $$
        or almost surely (a.s.) if
        $$
            \mathbb{P}(\{\omega \colon \forall \, k>0, \exists \, n_0 \text{ such that } \forall \,t>n_0, \, |M_t|<kg(t)\}) = 1.
        $$
        Similar interpretations will hold for $O$ and $\Omega$ notations.

        In this paper we consider $\ell(k)$ constant, i.e.\ $\mathbb{P}(X=i) = \beta(\alpha)i^{-\alpha}$, $i \in \{1,2,\dots\}$ (power-law distribution), where $\beta(\alpha)$ is the normalising constant (throughout this paper we will use $\beta$ instead of $\beta(\alpha)$ when $\alpha$ is clear from context).
        Take note that $X$ has infinite expectation for every $\alpha \in (1,2]$.

        As mentioned in the introductory section, the typical approach to the study of the limiting behaviour of the random degree sequence $\{r_k(t)\}_{k=1}^\infty$ relies on the presence of the property of concentration towards a limiting deterministic vector. The following theorem shows that concentration does not hold in the case of $\alpha \in (1,2)$.
        
        \begin{theorem}\label{thm:main}
            Let $\alpha \in (1,2)$ and consider the PARID-model with parameters $\delta=0$ and $X$ which follows a power-law distribution with exponent $\alpha$. 
            Then for any fixed $k\ge 1$ there exists no constant $a \in [0,1]$ such that $\ratdeg_k(t)\stackrel{p}{\to}a$ as $t\to \infty$.
        \end{theorem}

        On the other hand, if $\alpha = 2$, concentration holds
        as stated in the following theorem. It follows that in this case Conjecture 1.4 in \cite{MR2480915} holds.

        \begin{theorem}\label{thm:main2}
            Consider the PARID-model with parameters $\delta=0$ and $X$ with power-law distribution of exponent $2$. Then for any $k\ge 1$ we have $r_k(t) \overset{p}{\to} b_k$, where
            \begin{align}
                b_k= \frac{2\beta(2)}{k(k+1)(k+2)}\sum_{i=1}^{k}\left(1+\frac{1}{i}\right).
            \end{align}
        \end{theorem}

        While we defer the proofs of the above two theorems to the following section, in the remainder of this sections we highlight the difficulties which arise for the case $\alpha\in(1,2]$.
        Indeed, while in the case $\alpha>2$ we know $\Ex{X}$ is bounded, this is no longer true when $\alpha\in(1,2]$.
        This prevents the direct use of the expectation of $X$ for our aims. Hence, we develop a technique to cope with this issue. Specifically, in order to avoid having to work with random variables with infinite expectation, we truncate $X$ at a sufficiently large value. We underline that a key-point in the proofs is that the chosen truncation point depends on $t$ in a suitable way. This means that we condition on $X$ not realising values above the truncation point. 
     
        The leading idea of the analysis of the process is that inserting an excessive number of edges repeatedly into the graph will cause a significant upset in the degree distribution, which could potentially lead to non concentration.
        We expect that for certain values of $\alpha$ there will be several steps where the number of edges inserted is of the same order of magnitude as the total number of edges inserted so far. Indeed, for $\alpha \in (1,2)$ the probability that such large number of edges is inserted is bounded away from zero, leading to the non concentration of the degree sequence. If instead $\alpha = 2$ such probability is $o(1)$, which entails the concentration of the degree sequence. 
        The two cases correspond in fact to the two Theorems \ref{thm:main} and \ref{thm:main2}.

        In comparison \cite{MR2480915}, also starts with truncating $X$ and establishing the degree sequence for the truncated model. However, they truncate at a level where some edges are missing and show that these missing edges have an insignificant effect on the degree sequence, when $\alpha>2$. However, the number of missing edges in the truncated model becomes too large to apply this argument when $\alpha\le 2$. Increasing the truncation threshold, would lead to fewer missing edges, but then their analysis of the truncated model no longer holds. Therefore, we raise the truncation threshold and provide an alternative analysis for the degree sequence of the truncated model.

    \section{Proofs}\label{proofs}

        We first prove the necessary lemmas for Theorems~\ref{thm:main} and \ref{thm:main2} in Sections~\ref{lemm1} and \ref{lemm2}, respectively, while Sections~\ref{th:ma:pr} and \ref{th:ma2:pr} contain the proofs of the theorems. For the sake of clarity, Figures \ref{fig} and \ref{twofig} show sketches of the relationships between lemmas in the proofs of Theorems \ref{thm:main} and \ref{thm:main2}, respectively.

        \subsection{Proof of Theorem~\ref{thm:main}}

            \subsubsection{Lemmas}\label{lemm1}

                \begin{figure}
                    \centering
                    \includegraphics[scale=.5]{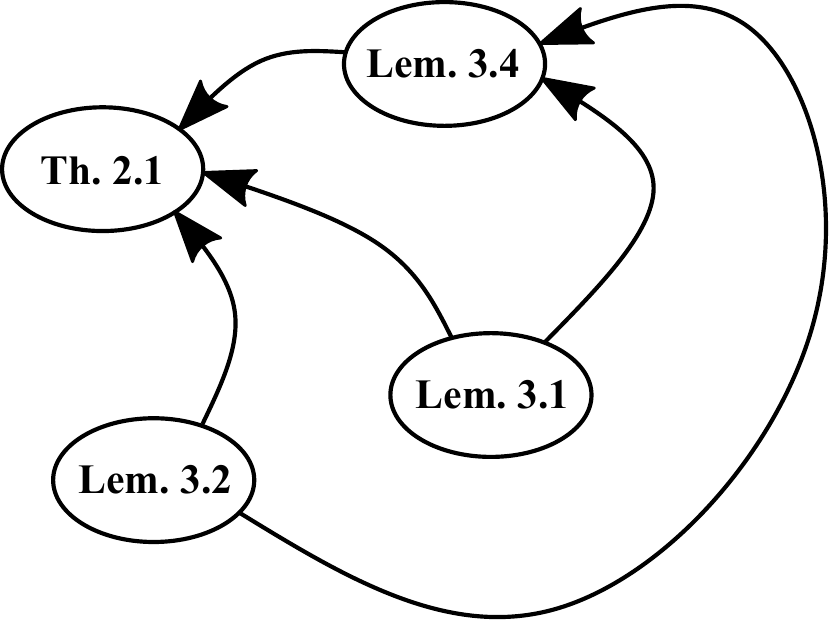}
                    \caption{\label{fig}Relationship        between lemmas in the proof of Theorem~\ref{thm:main}.}
                \end{figure}

                We start by focusing on the number of edges in the graph. It will turn out that the ratio of  edges to vertices is not concentrated as $t\to \infty$. Later we will exploit this property in the proof of Theorem~\ref{thm:main}.

                For $\alpha\in(1,2)$ define
                \begin{align}\label{eq:def:Cs}
                    c:=c(\alpha)=\left(\frac{\alpha-1}{8\beta}\right)^{\frac{1}{1-\alpha}} \quad \mbox{and} \quad C:=C(\alpha,t)=\frac{\beta (2c)^{2-\alpha}}{(2-\alpha)\PPo{X< ct^{1/(\alpha-1)}+1}}.
                \end{align}

                As previously mentioned in Section \ref{results}, our aim is to describe a meaningful truncation of the random variable $X$ depending on $t$ in a suitable manner. 
                In Lemma~\ref{lem:upperbound} below, we introduce such a truncation and establish that it is unlikely for the graph to have a very large number of edges by time $t$.

                \begin{lemma}\label{lem:upperbound}
                    Let $\alpha\in(1,2)$. 
                    For any positive real $z$ and $t\ge c^{1-\alpha}$ we have
                    $$
                        \PPo{\sum_{i=1}^t X_i < z C t^{1/(\alpha-1)}}> \frac{7}{8}-\frac{1}{z}.
                    $$
                \end{lemma}

                \begin{proof}
                    Since the function $\beta x^{-\alpha}$ is monotone decreasing, for any $y>1$ we have 
                    $$
                        \PPo{X\ge y} \le \beta \int_{y-1}^{\infty}x^{-\alpha}dx.
                    $$
                    Therefore
                    $$
                        \PPo{X\ge c t^{1/(\alpha-1)}+1}\le \beta \int_{c t^{1/(\alpha-1)}}^{\infty}x^{-\alpha}dx=\frac{\beta}{\alpha-1}c^{1-\alpha}t^{-1}\stackrel{\eqref{eq:def:Cs}}{=}(8t)^{-1}.
                    $$
                    Let $\mathcal{G}$ be the event that for every $1\le \tau \le t$ we have $X_\tau<ct^{1/(\alpha-1)}+1$. Then by the union bound on the complement of $\mathcal{G}$ we have 
                    \begin{equation}\label{eq:nolargedeg}
                        \PPo{\mathcal{G}}\ge 7/8.
                    \end{equation}
	
                	Denote by $Y$ the random variable $X$ truncated at $ct^{1/(\alpha-1)}+1$, that is for every $i< ct^{1/(\alpha-1)}+1$ we have 
                    $$
                        \PPo{Y=i}=\frac{\beta}{\PPo{X< ct^{1/(\alpha-1)}+1}} i^{-\alpha}=\beta' i^{-\alpha},
                    $$
                    where $\beta'=\beta \PPo{X< ct^{1/(\alpha-1)}+1}^{-1}$.
                    Let $\{i_1,\ldots,i_t\}$ be a set of integers such that for every $1\le \tau \le t$ we have $1\le i_\tau < ct^{1/(\alpha-1)}+1$.
                    Then due to the independence of $X_\tau$, for $1\le \tau \le t$ we have 
                    $$
                        \PPo{\bigcap_{\tau=1}^t X_\tau=i_\tau\mid \mathcal{G}}=\prod_{\tau=1}^t\PPo{Y=i_\tau}.
                    $$
                    Let $\mathcal{S}$ be the support of $Y$.
                    Since $x^{1-\alpha}$ is monotone decreasing,
                    \begin{align*}
                        \Ex{Y}&=\sum_{i\in \mathcal{S}}\beta'i^{1-\alpha}=\beta'\left(1+\sum_{i\in \mathcal{S}\backslash\{1\}}i^{1-\alpha}\right)\\
                        &\le \beta'\left(1+\int_{1}^{ct^{1/(\alpha-1)}+1} x^{1-\alpha} dx\right)
                        =\beta'\left(1+\frac{(ct^{1/(\alpha-1)}+1)^{2-\alpha}}{2-\alpha}-\frac{1}{2-\alpha}\right)\\
                        & \stackrel{t\ge c^{1-\alpha}}{\le} \beta' \frac{(2ct^{1/(\alpha-1)})^{2-\alpha}}{2-\alpha}\stackrel{\eqref{eq:def:Cs}}{=}Ct^{\frac{2-\alpha}{\alpha-1}}.
                    \end{align*}
                    Therefore
                    $$
                        \Ex{\sum_{i=1}^{t} X_i\mid \mathcal{G}}\le Ct^{1/(\alpha-1)}
                    $$
                    and by Markov's inequality
                    $$
                        \PPo{\sum_{i=1}^{t} X_i< zCt^{1/(\alpha-1)}\mid \mathcal{G}}> 1-\frac{1}{z}.
                    $$
                    Together with \eqref{eq:nolargedeg} we have
                    $$
                        \PPo{\sum_{i=1}^{t} X_i < zCt^{1/(\alpha-1)}}
                        > \frac{7}{8}\left(1-\frac{1}{z}\right)\ge \frac{7}{8}-\frac{1}{z}.
                    $$
                \end{proof}

                In the next lemma we investigate the possibility of inserting a number of edges of the same order of magnitude as the total number of edges inserted so far (given in Lemma \ref{lem:upperbound}). In particular,
                the probability that this happens is $\Omega(t^{-1})$.

                \begin{lemma}\label{lem:largevalueprob}
                    Let $\alpha\in(1,2)$.
                    For any positive integer $t$ and positive real $\gamma$ satisfying $\gamma C t^{1/(\alpha-1)}>1$, we have
                    $$
                        \PPo{X\ge \gamma C  t^{1/(\alpha-1)}}\ge \beta\frac{(2\gamma C)^{1-\alpha}}{\alpha-1}t^{-1}.
                    $$
                \end{lemma}

                \begin{proof}
                    Since $x^{-\alpha}$ is monotone decreasing we have
                    $$
                        \PPo{X\ge \gamma C  t^{1/(\alpha-1)}}\ge \beta \int_{\gamma C  t^{1/(\alpha-1)}+1}^{\infty} x^{-\alpha} dx \ge \beta\frac{(2\gamma C)^{1-\alpha}}{\alpha-1}t^{-1}.
                    $$
                \end{proof}

                \begin{remark}
                    Roughly speaking, Lemma \ref{lem:largevalueprob} tells us that if we consider the addition of $\Omega(t)$ vertices to the graph, the probability that there exists a step where the number of newly added edges is of the same order of magnitude as the total number of edges inserted so far, is bounded away from zero.
                \end{remark}

                Recall
                $\numdeg_i(t)$ is the number of vertices of degree $i$ at time $t$ and let
                \begin{equation}\label{eq:Q}
                    \smalldeg_k(t)=\sum_{i=1}^k \numdeg_i(t)
                \end{equation}
                be the number of vertices of degree less than or equal to $k$.
                Denote by $\mathcal{L}_k(b,t,\eta)$, $b,\eta \in \mathbb{R}_+$, the event that there exists a step $t<\tau \le (1+\eta)t$ such that $\smalldeg_k(\tau)< b\tau-2\eta \tau$.
                If the proportion of vertices with degree at most $k$ were concentrated, then we would expect $Q_k(\tau)$ close to $b\tau$, for some $b$ when $\tau$ is sufficiently large. The event $\mathcal{L}_k(b,t,\eta)$ represents the opposite of this: it indicates that there is some $t<\tau\le (1+\eta)t$ such that $Q_k(\tau)$ is far from $b\tau$ for some step $\tau$. In the following lemma we show that $\mathcal{L}_k(b,t,\eta)$ holds with probability bounded away from zero.

                \begin{lemma}\label{lem:lowerbound}
                    Consider $\alpha\in (1,2)$ and a fixed $k\ge 1$.
                    Assume $0<b\le 1$ and $0<\eta\le b/8$. For any $\gamma> 32$ which satisfies
                    \begin{equation}\label{eq:gamma}
                        \exp\left(-\frac{(1-32/\gamma)^2 \gamma k }{64}\right) \le b  -4 \eta,
                    \end{equation} 
                    and every positive integer $t$ such that $\gamma C t^{1/(\alpha-1)}>1$ and $\eta t> c^{1-\alpha}$
                    we have
                    $$
                        \PPo{\mathcal{L}_k(b,t,\eta)}\ge \frac{4 \eta^3 t}{1+4\eta^2 t}\cdot \frac{9}{16} \beta\frac{(2k\gamma C)^{1-\alpha}}{\alpha-1}.
                    $$
                \end{lemma}

                Note that \eqref{eq:gamma} is satisfied when
                $(1-32/\gamma)^2\gamma \ge - \frac{64}{k}\log(b-4\eta)$. Hence there exists a $\gamma> 32$ such that \eqref{eq:gamma} holds. 

                \begin{proof}
                    For every integer $\tau$ satisfying $t< \tau \le (1+\eta)t$ denote by $\mathcal{E}_\tau$ the intersection of the following three events:
                    \begin{itemize}
                        \item $\left\{\sum_{j=1}^t X_j < 8C t^{1/(\alpha-1)}\right\}$;
                        \item $\left\{\sum_{j=t+1}^{\tau-1} X_j < 8 C (\eta t)^{1/(\alpha-1)}\right\}$;
                        \item $\left\{X_\tau \ge k\gamma Ct^{1/(\alpha-1)}\right\}$.
                    \end{itemize}
                    Since $k\gamma \ge 32 \ge 8 \eta^{1/(\alpha-1)}$ the events $\{\mathcal{E}_\tau:t<\tau\le (1+\eta)t\}$ are mutually exclusive, thus
                    \begin{equation}\label{eq:prob}
                        \PPo{\mathcal{L}_k(b,t,\eta)}\ge \sum_{\tau=t+1}^{(1+\eta)t}\PPo{\mathcal{L}_k(b,t,\eta) \mid \mathcal{E}_\tau}\PPo{\mathcal{E}_{\tau}}.
                    \end{equation}

                    We start by considering the probability
                    $\PPo{\mathcal{E}_{\tau}}$. Due to the independence of the $X_j$ we have
                    \begin{align*}
                        \PPo{\mathcal{E}_{\tau}}
                        &= \PPo{\sum_{j=1}^t X_j < 8C t^{1/(\alpha-1)}, \sum_{j=t+1}^{\tau-1} X_j < 8 C (\eta t)^{1/(\alpha-1)}, X_\tau\ge  k\gamma C t^{1/(\alpha-1)}}\\
                        &= \PPo{\sum_{j=1}^t X_j < 8C t^{1/(\alpha-1)}} \PPo{\sum_{j=t+1}^{\tau-1} X_j < 8 C (\eta t)^{1/(\alpha-1)}} \PPo{X_\tau\ge k\gamma C t^{1/(\alpha-1)}}\\
                        &\ge\PPo{\sum_{j=1}^t X_j < 8C t^{1/(\alpha-1)}} \PPo{\sum_{j=1}^{\eta t} X_j < 8 C (\eta t)^{1/(\alpha-1)}} \PPo{X_\tau \ge k\gamma C t^{1/(\alpha-1)}}.
                    \end{align*}
                    Lemma~\ref{lem:upperbound} implies that the first two terms of this product are each at least $3/4$. 
                    We deduce from Lemma~\ref{lem:largevalueprob} that the last term can be bounded from below by
                    $$
                        \beta\frac{(2k\gamma C)^{1-\alpha}}{\alpha-1}t^{-1}.
                    $$
                    Therefore
                    \begin{equation}\label{eq:lbjp}
                        \PPo{\mathcal{E}_{\tau}}\ge \frac{9}{16} \beta \frac{(2k\gamma C)^{1-\alpha}}{\alpha-1}t^{-1}.
                    \end{equation}

                    All that is left is to provide a lower bound on $\PPo{\mathcal{L}_k(b,t,\eta) \mid \mathcal{E}_\tau}$. 
                    We will use that
                    \begin{equation}\label{aa}
                        \PPo{\mathcal{L}_k(b,t,\eta) \mid \mathcal{E}_\tau} \ge \PPo{\smalldeg_k(\tau)< b\tau-2\eta \tau \mid \mathcal{E}_{\tau}},
                    \end{equation}
                    for every fixed $t<\tau\le (1+\eta)t$.
                    Denote by $U$ the set of indices of vertices with degree at most $k$ immediately before step $\tau$. For $i\in U$ let $I_{v_i}$ be the indicator random variable that $v_i$ still has degree at most $k$ at the end of step $\tau$. Since $\mathcal{E}_\tau$ holds, we have $X_\tau \ge k\gamma C t^{1/(\alpha-1)}>k$ and thus no vertex with degree at most $k$ is added in step $\tau$. Therefore
                    $$
                        \smalldeg_k(\tau)=\sum_{i\in U}I_{v_i}.
                    $$

                    Consider $m,\ell$ positive integers and $\mathrm{d}$ a vector of length $\tau$ consisting of positive integers. Denote the elements of the vector $\mathrm{d}$ by $\mathrm{d}_0,\ldots,\mathrm{d}_{\tau-1}$. Let $\mathcal{M}_{m,\ell,\mathrm{d}}$ be the event that $\sum_{j=1}^{\tau-1}X_j=m$,  $X_\tau=\ell$ and $d_i(\tau-1)=\mathrm{d}_i$, for every $0\le i \le \tau-1$.

                    Fix an arbitrary $m,\ell$ and $\mathrm{d}$ satisfying $\PPo{\mathcal{M}_{m,\ell,\mathrm{d}}\mid \mathcal{E}_{\tau}}>0$. 
                    Note that $\Ex{\smalldeg_k(\tau)\mid \mathcal{M}_{m,\ell,\mathrm{d}}, \mathcal{E}_{\tau}}=\Ex{\smalldeg_k(\tau)\mid \mathcal{M}_{m,\ell,\mathrm{d}}}$ and $\Var{\smalldeg_k(\tau)\mid \mathcal{M}_{m,\ell,\mathrm{d}}, \mathcal{E}_{\tau}}=\Var{\smalldeg_k(\tau)\mid \mathcal{M}_{m,\ell,\mathrm{d}}}$, and we will establish these values in the following.
                    For any $i\in U$
                    \begin{align*}
                        & \PPo{I_{v_i}=1\mid \mathcal{M}_{m,\ell,\mathrm{d}}} = \PPo{B_{\ell, \frac{\mathrm{d}_i}{2m}} \le k-\mathrm{d}_i}
                        \le \PPo{B_{\ell, \frac{1}{2m}} \le k-\mathrm{d}_i}
                        \le \PPo{B_{\ell, \frac{1}{2m}} \le k}, 
                    \end{align*}
                    where $B_{\ell,p}$ is a binomial random variable with parameters $\ell$ and $p$.
                    Using $\PPo{\mathcal{M}_{m,\ell,\mathrm{d}}\mid \mathcal{E}_{\tau}}>0$ we have $m\le16 C t^{1/(\alpha-1)}$ and $\ell\ge k\gamma C t^{1/(\alpha-1)}$. Therefore
                    \begin{align*}
                        \Ex{I_{v_i}\mid \mathcal{M}_{m,\ell,\mathrm{d}}}
                        &=\PPo{I_{v_i}=1\mid \mathcal{M}_{m,\ell,\mathrm{d}}}
                        \le \PPo{B_{\ell, \frac{1}{2m}} \le k}\\
                        &\le \PPo{B_{k\gamma C t^{1/(\alpha-1)}, \frac{1}{32 C t^{1/(\alpha-1)}}}\le k}\le \exp\left(-\frac{(1-32/\gamma)^2 \gamma k }{64}\right), \label{eq:index}
                    \end{align*}
                    where in the last step we used the Chernoff bound $\PPo{B_{\ell,p}\le (1-\delta)\ell p}\le \exp(-\delta^2 \ell p/2)$.
                    We deduce
                    \begin{equation}\label{eq:expsmalldeg}
                        \Ex{\smalldeg_k(\tau) \mid \mathcal{M}_{m,\ell,\mathrm{d}}}=\Ex{\sum_{u\in U} I_{v_i} \mid \mathcal{M}_{m,\ell,\mathrm{d}}}\le  \tau \exp\left(-\frac{(1-32/\gamma)^2 \gamma k }{64}\right) \stackrel{\eqref{eq:gamma}}{\le} b \tau -4 \eta \tau,
                    \end{equation}
                    as $|U|\le \tau$.
                    In addition, for every $i,i'\in U$ we have
                    \begin{equation}\label{eq:neg_conv}
                        \PPo{I_{v_i}=0 \mid I_{v_{i'}}=0,\mathcal{M}_{m,\ell,\mathrm{d}}}\le \PPo{I_{v_i}=0 \mid \mathcal{M}_{m,\ell,\mathrm{d}}} ,
                    \end{equation}
                    as any edge inserted in step $\tau$ can increase the degree of at most one of $u$ and $u'$. 
                    Then
                    \begin{align*}
                        &\PPo{I_{v_i}=0\mid \mathcal{M}_{m,\ell,\mathrm{d}}}\PPo{I_{v_{i'}}=0\mid \mathcal{M}_{m,\ell,\mathrm{d}}}\\
                        &\quad=1-\PPo{I_{v_i}=1\mid \mathcal{M}_{m,\ell,\mathrm{d}}}-\PPo{I_{v_{i'}}=1\mid \mathcal{M}_{m,\ell,\mathrm{d}}}+\PPo{I_{v_i}=1\mid \mathcal{M}_{m,\ell,\mathrm{d}}}\PPo{I_{v_{i'}}=1\mid \mathcal{M}_{m,\ell,\mathrm{d}}}
                    \end{align*}
                    and
                    \begin{align*}
                        &\PPo{I_{v_i}=0,I_{v_{i'}}=0\mid \mathcal{M}_{m,\ell,\mathrm{d}}}\\
                        &\quad =1-\PPo{I_{v_i}=1\mid \mathcal{M}_{m,\ell,\mathrm{d}}}-\PPo{I_{v_{i'}}=1\mid \mathcal{M}_{m,\ell,\mathrm{d}}}+\PPo{I_{v_i}=1,I_{v_{i'}}=1\mid \mathcal{M}_{m,\ell,\mathrm{d}}}.
                    \end{align*}
                    Therefore \eqref{eq:neg_conv} is equivalent to
                    $$
                        \PPo{I_{v_i}=1 \mid I_{v_{i'}}=1,  \mathcal{M}_{m,\ell,\mathrm{d}}}\le \PPo{I_{v_i}=1 \mid \mathcal{M}_{m,\ell,\mathrm{d}}},
                    $$
                    implying that the indicator random variables are negatively correlated. Since $|U|\le \tau$  we have
                    \begin{equation}\label{eq:varsmalldeg}
                        \Var{Q_k(\tau)\mid \mathcal{M}_{m,\ell,\mathrm{d}}}\le \tau.
                    \end{equation}

                    By Cantelli's inequality using \eqref{eq:expsmalldeg} and \eqref{eq:varsmalldeg} we  have
                    \begin{align*}
                        \PPo{\smalldeg_k(\tau)\ge b\tau-2\eta \tau \mid \mathcal{M}_{m,\ell,\mathrm{d}}}
                        &\le \mathbb{P}\Bigl[ \smalldeg_k(\tau)\ge \Ex{\smalldeg_k(\tau)\mid \mathcal{M}_{m,\ell,\mathrm{d}}}+2\eta \tau \mid \mathcal{M}_{m,\ell,\mathrm{d}}\Bigr]\\
                        &\le \frac{1}{1+\frac{4\eta^2 \tau^2}{\Var{\smalldeg_k(\tau)\mid \mathcal{M}_{m,\ell,\mathrm{d}}}}}\le \frac{1}{1+\frac{4\eta^2 \tau^2}{\tau}}\le \frac{1}{1+4\eta^2 t}.
                    \end{align*}

                    Since the above bound holds for every choice of $m,\ell,\mathrm{d}$ satisfying $\PPo{\mathcal{M}_{m,\ell,\mathrm{d}}\mid \mathcal{E}_{\tau}}>0$ we also have
                    \begin{align*}
                        \PPo{\smalldeg_k(\tau)\ge b\tau-2\eta \tau \mid \mathcal{E}_{\tau}}
                        &=\sum_{\substack{m,\ell \in \mathbb{N},\mathrm{d}\in \mathbb{N}^{\tau}\\\PPo{\mathcal{M}_{m,\ell,\mathrm{d}}\mid \mathcal{E}_{\tau}}>0}}\PPo{\smalldeg_k(\tau)\ge b\tau-2\eta \tau \mid \mathcal{M}_{m,\ell,\mathrm{d}},\mathcal{E}_{\tau}}\PPo{\mathcal{M}_{m,\ell,\mathrm{d}}\mid \mathcal{E}_{\tau}}\\
                        &=\sum_{\substack{m,\ell \in \mathbb{N},\mathrm{d}\in \mathbb{N}^{\tau}\\\PPo{\mathcal{M}_{m,\ell,\mathrm{d}}\mid \mathcal{E}_{\tau}}>0}}\PPo{\smalldeg_k(\tau)\ge b\tau-2\eta \tau \mid \mathcal{M}_{m,\ell,\mathrm{d}}}\PPo{\mathcal{M}_{m,\ell,\mathrm{d}}\mid \mathcal{E}_{\tau}}\\
                        &\le \frac{1}{1+4\eta^2 t}.
                    \end{align*}

                    By looking at the complementary event and by \eqref{aa},
                    \begin{equation}\label{eq:lower2}
                        \frac{4\eta^2 t}{1+4\eta^2 t} \le \PPo{\smalldeg_k(\tau)< b\tau-2\eta \tau \mid \mathcal{E}_{\tau}} \le \PPo{\mathcal{L}_k(b,t,\eta)\mid \mathcal{E}_\tau}.
                    \end{equation}
                    Plugging \eqref{eq:lbjp} and \eqref{eq:lower2} into \eqref{eq:prob} gives
                    \begin{equation*}
                        \frac{4\eta^2 t}{1+4\eta^2 t}\cdot \frac{9}{16} \beta \frac{(2 k\gamma C)^{1-\alpha}}{\alpha-1}\eta \le \PPo{\mathcal{L}_k(b,t,\eta)},
                    \end{equation*}
                    as required.
                \end{proof}

            \subsubsection{Proof of Theorem~\ref{thm:main}}\label{th:ma:pr}

                Fix $\alpha\in(1,2)$, $k\ge 1$, and recall \eqref{eq:def:Cs} and \eqref{eq:Q}. Assume for contradiction that there exists an $0\le a\le 1$ such that $\ratdeg_k(t)\stackrel{p}{\to}a$ as $t\to \infty$. Then, by definition, for every $\zeta>0$ we have $\PPo{|\ratdeg_k(t)-a|\ge \zeta}\to 0$ as $t\to \infty$. 
                We will consider separately the two cases $a>0$ and $a=0$.

                \bigskip

                \noindent \underline{Case $a>0$:}

                \medskip

                Set $\eta=a/8$. 
                Note that if there exists a step $t< \tau \le (1+\eta)t$ such that $\smalldeg_k(\tau)< a\tau-2\eta \tau$, then since in every step at most one vertex of degree at most $k$ is created and because $a\ge \eta$ we have
                $$
                    \numdeg_k((1+\eta)t)\le \smalldeg_k((1+\eta)t)< a\tau-2\eta \tau +(1+\eta)t-\tau < a\tau -\eta \tau \le (a-\eta)(1+\eta)t
                $$
                and thus
                $$
                    \ratdeg_k((1+\eta)t)=\frac{\numdeg_k((1+\eta)t)}{(1+\eta)t+1}<a-\eta.
                $$
                Recall that $\mathcal{L}_k(a,t,\eta)$, $a,\eta \in \mathbb{R}_+$, is the event that there exists a step $t<\tau \le (1+\eta)t$ such that $\smalldeg_k(\tau)< a\tau-2\eta \tau$. Therefore
                \begin{equation}\label{eq:nonconc}
                    \PPo{\mathcal{L}_k(a,t,\eta)}\le \PPo{\ratdeg_k((1+\eta)t)< a-\eta}.
                \end{equation}

                Select $\gamma > 32$ which satisfies \eqref{eq:gamma}.
                Then for $t\ge \max\{(2\eta)^{-2},\eta^{-1}c^{1-\alpha},(\gamma C)^{1-\alpha}\}$
                we have
                \begin{align*}
                    \PPo{|r_k((1+\eta)t)-a|\ge  \eta} & \ge \PPo{\ratdeg_k((1+\eta)t)< a-\eta} \\
                    &\stackrel{\eqref{eq:nonconc}}{\ge} \PPo{\mathcal{L}_k(a,t,\eta)}
                    \stackrel{\text{Lem.}~\ref{lem:lowerbound}}{\ge} \frac{4\eta^3 t}{1+4\eta^2 t} \cdot \frac{9}{16} \beta \frac{(2k\gamma C)^{1-\alpha}}{\alpha-1}\\
                    &\stackrel{t\ge (2\eta)^{-2}}{\ge} \frac{9 \eta}{32} \beta \frac{(2k\gamma C)^{1-\alpha}}{\alpha-1},
                \end{align*}
                a contradiction, as 
                $$
                    \lim_{t\to\infty} \frac{9 \eta}{32} \beta \frac{(2k\gamma C(\alpha,t))^{1-\alpha}}{\alpha-1}>0,
                $$
                where $t \mapsto C(\alpha,t)$ is decreasing as $t \to \infty$ and
                \begin{equation}\label{eq:lim:C}
                    C_\infty= \lim_{t\to\infty} C(\alpha,t)=\frac{ (2c)^{2-\alpha}\beta}{2-\alpha}.
                \end{equation}

                \smallskip

                \noindent \underline{Case $a=0$:}

                \medskip

                Let $\mathcal{J}_t$ be the event that 
                $$
                    \sum_{j=1}^t X_j\ge 16 C t^{1/(\alpha-1)}
                $$
                and $\mathcal{K}_t$ be the event that 
                $$
                    \sum_{j=t+1}^{2t} X_j\le 8 C t^{1/(\alpha-1)}.
                $$

                For $t< j \le 2t$, let $I_j$ be the indicator random variable for the event the initial degree of vertex $v_j$ is $k$ and the degree of vertex $v_j$ remains unchanged until step $2t$. 
                Since $1-x\ge e^{-2x}$ when $x\le 1/2$,
                for $t\ge \left(\frac{k}{8C_\infty}\right)^{\alpha-1}$,
                \begin{align}\label{eq:indprob}
                    \PPo{I_j=1\mid \mathcal{J}_t,\mathcal{K}_t}&\ge \PPo{X_j=k \mid \mathcal{J}_t,\mathcal{K}_t} \left(1-\frac{k}{16 C t^{1/(\alpha-1)}}\right)^{8C t^{1/(\alpha-1)}} \\ & \ge \PPo{X_j=k \mid \mathcal{K}_t} e^{-k}, \notag
                \end{align}
                where the last inequality also uses the fact that $X_j$ and $\mathcal{K}_t$ are independent of $\mathcal{J}_t$.

                Since $1/(\alpha-1)>1$, Bernoulli's inequality implies 
                \begin{align}\label{bb}
                    8Ct^{1/(\alpha-1)}-k&=8C(t-1)^{1/(\alpha-1)}\left(1+\frac{1}{t-1}\right)^{1/(\alpha-1)}-k\\
                    &\ge 8C (t-1)^{1/(\alpha-1)}+\frac{8C}{\alpha-1}(t-1)^{1/(\alpha-1)-1}-k\notag \\
                    &\ge 8C (t-1)^{1/(\alpha-1)},\notag
                \end{align}
                where the last inequality holds if $t \ge 1+\left(\frac{k(\alpha-1)}{8C_\infty}\right)^{\frac{\alpha-1}{2-\alpha}}$.
                Therefore, since the $X_i$ are independent
                \begin{align}\label{cc}
                    \PPo{X_j=k \mid \mathcal{K}_t}&\ge \PPo{X_j=k}\PPo{\sum_{\substack{\ell=t+1\\\ell\neq j}}^{2t}X_\ell\le 8Ct^{1/(\alpha-1)}-k}\\
                    &\stackrel{\eqref{bb}}{\ge} \PPo{X_j=k}\PPo{\sum_{\substack{\ell=t+1\\\ell\neq j}}^{2t}X_\ell\le 8C(t-1)^{1/(\alpha-1)}}\notag\\
                    &\stackrel{\text{Lem.}~\ref{lem:upperbound}}{\ge} \frac{3}{4}\beta k^{-\alpha},\notag
                \end{align}
                where the last inequality holds if $t-1 \ge c^{1-\alpha}$. Hence \eqref{cc} holds for
                \begin{equation*}
                    t \ge \max \left(1+\left(\frac{k(\alpha-1)}{8C_\infty}\right)^{\frac{\alpha-1}{2-\alpha}}, 1+c^{1-\alpha} \right).
                \end{equation*}
                To combine \eqref{cc} with \eqref{eq:indprob} we assume that
                \begin{equation}\label{eq:tdef2}
                    t\ge \max\left\{\left(\frac{k}{8C_\infty}\right)^{\alpha-1},1+\left(\frac{k(\alpha-1)}{8C_\infty}\right)^{\frac{\alpha-1}{2-\alpha}}, 1+c^{1-\alpha}\right\},
                \end{equation}
                in order to obtain
                $$
                    \PPo{I_j=0\mid \mathcal{J}_t,\mathcal{K}_t}< 1-\frac{3\beta}{4 e^k}k^{-\alpha}.
                $$
                Let us introduce the random variable $A=t-\sum_{j=t+1}^{2t}I_j$ counting the number of vertices $v_j$, $t<j\le 2t$, whose degree is not $k$ at some point during the first $2t$ steps. We have
                $$
                    \Ex{A \mid \mathcal{J}_t,\mathcal{K}_t}< \left(1-\frac{3\beta}{4 e^k}k^{-\alpha}\right)t
                $$
                and by Markov's inequality
                $$
                    \PPo{A \ge \left(1-\frac{\beta}{4 e^k}k^{-\alpha}\right)t \mid \mathcal{J}_t,\mathcal{K}_t}\le \frac{1-3\beta k^{-\alpha}/(4 e^k)}{1-\beta k^{-\alpha}/(4 e^k)}=1-\frac{2\beta k^{-\alpha}}{4 e^k-\beta k^{-\alpha}}.
                $$
                Since  $R_k(2t) \ge \sum_{j={t+1}}^{2t}I_j = t-A$, 
                \begin{equation}\label{eq:degkcond}
                    \frac{2 \beta k^{-\alpha}}{4 e^k-\beta k^{-\alpha}}<\PPo{\numdeg_k(2t)>\frac{\beta}{4 e^k}k^{-\alpha}t \mid \mathcal{J}_t,\mathcal{K}_t}<\PPo{\numdeg_k(2t)>\frac{\beta}{12 e^k}k^{-\alpha}(2t+1) \mid \mathcal{J}_t,\mathcal{K}_t} .
                \end{equation}
                For $t$ satisfying \eqref{eq:tdef2}, Lemma~\ref{lem:largevalueprob} implies that
                $$
                    \PPo{\mathcal{J}_t}\ge 1-\left(1-\beta\frac{(32C)^{1-\alpha}}{\alpha-1}t^{-1}\right)^t \ge 1-\exp\left(-\beta\frac{(32C)^{1-\alpha}}{\alpha-1}\right),
                $$
                and by Lemma~\ref{lem:upperbound} $\PPo{\mathcal{K}_t}\ge 3/4$. 
                Together with \eqref{eq:degkcond} we have
                \begin{align*}
                    \PPo{\ratdeg_k(2t)\ge \frac{\beta}{12 e^k}k^{-\alpha}}
                    &= \PPo{\numdeg_k(2t)\ge\frac{\beta}{12 e^k}k^{-\alpha}(2t+1) }\\
                    &\ge \frac{3}{4} \cdot \frac{2\beta k^{-\alpha}}{4 e^k-\beta k^{-\alpha}}\left(1-\exp\left(-\beta\frac{(16C)^{1-\alpha}}{\alpha-1}\right)\right),
                \end{align*}
                a contradiction as 
                $$
                    \lim_{t\to \infty} \frac{3}{4} \cdot \frac{2\beta k^{-\alpha}}{4 e^k-\beta k^{-\alpha}}\left(1-\exp\left(-\beta\frac{(16C(\alpha,t))^{1-\alpha}}{\alpha-1}\right)\right)> 0,
                $$
                as $\eqref{eq:lim:C}$ is strictly positive.

            \subsection{Proof of Theorem~\ref{thm:main2}}

                Recall Theorem~\ref{thm:main2} regards the case $\alpha=2$; the same will hold for all the following lemmas.

                \subsubsection{Lemmas}\label{lemm2}

                    \begin{figure}
                        \centering
                        \includegraphics[scale=.5]{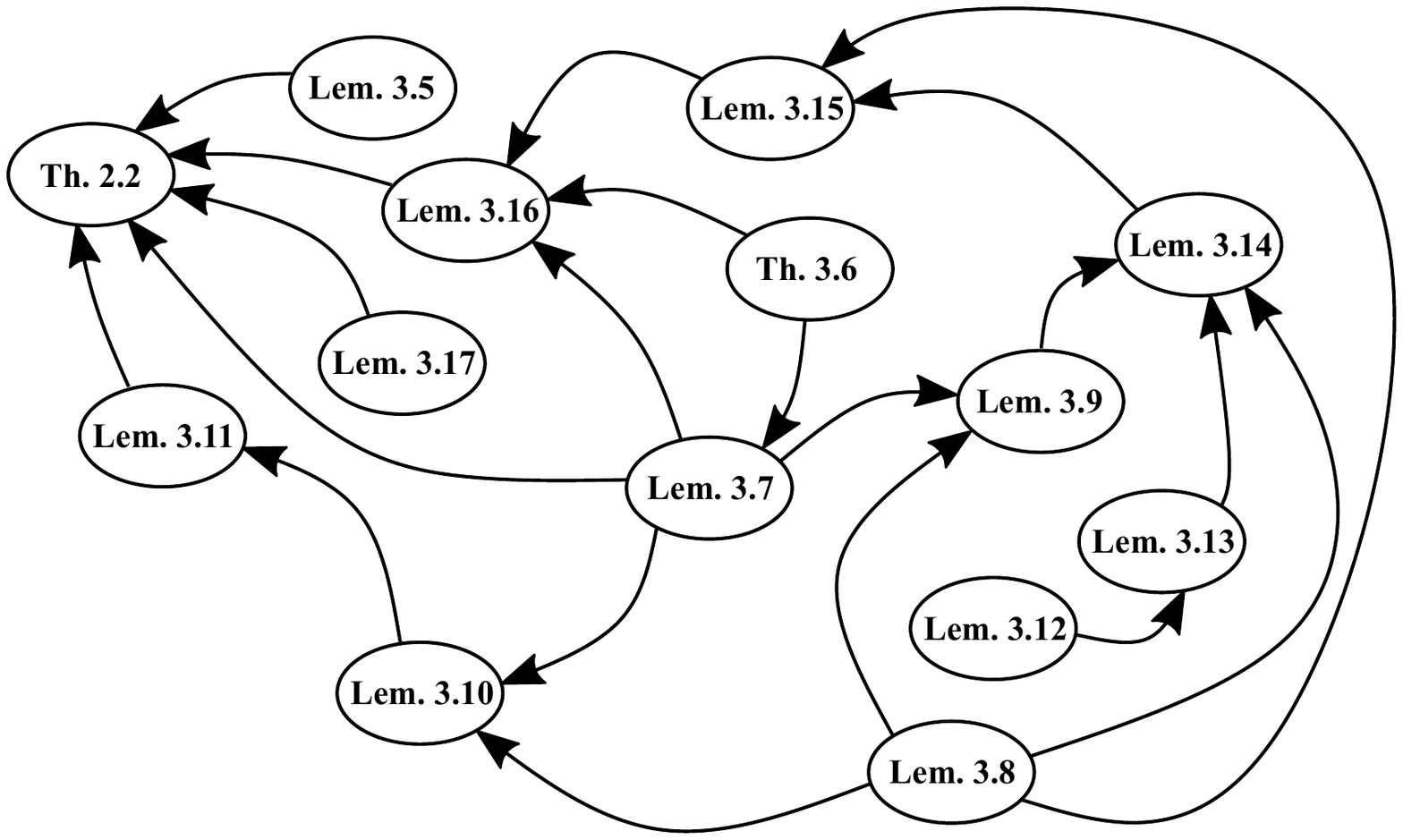}
                        \caption{\label{twofig}Relationship between lemmas in the proof of Theorem~\ref{thm:main2}. For example the lemmas \ref{trecinque},
                        \ref{lem:edgeconc},
                        \ref{lem:expR}, \ref{lem:marconc1}, \ref{lem:marconc2} are those directly recalled by the proof of Theorem~\ref{thm:main2}.}
                    \end{figure}

                    Let $t \ge 3$. 
                    We will truncate the random variable $X$ at $ t\log\log{t}+1$, so that whp none of the random variables $X_1,\ldots,X_t$ takes a value above the truncation point, and at the same time keeps the maximal number of edges added low and thus also the second moment of the truncated random variable.
                    Denote by $Z$ the random variable $X$ truncated at $ t\log\log{t}+1$, i.e.\ for every $i< t\log\log{t}+1$ we have 
                    $$
                        \PPo{Z=i}=\frac{\beta}{\PPo{X< t \log\log{t} + 1}} i^{-2}=\beta'' i^{-2},
                    $$
                    where $\beta''=\beta\PPo{X< t \log\log{t} + 1}^{-1}$.
                    In comparison with the previous case, when the truncation point was $c t^{1/(\alpha-1)}+1$, the current truncation point grows slower for any value of $\alpha\in(1,2)$, however it grows slightly quicker when $\alpha=2$.

                    We will often use the first and second moments of $Z$, 
                    \begin{align}
                        \Ex{Z}&= \sum_{i:1\le i < t\log\log{t}+1} \frac{\beta''}{i}=(1+o(1))\beta'' \log{t} \label{eq:firstmomZ}\\
                        \Ex{Z^2}&= \sum_{i:1\le i < t\log\log{t}+1} \beta''= (1+o(1))\beta''t\log\log{t}. \label{eq:secmomZ}
                    \end{align}

                    Consider a sequence $Z_1,Z_2,\ldots$, formed of independent copies of $Z$.  
                    Denote by $\mathcal{H}$ the event that for every $1\le \tau \le t$ we have $X_\tau< t \log\log{t}+1$. 

                    \begin{lemma}\label{trecinque}
                        For any event $\mathcal{D}$ we have 
                        $$
                            \PPo{\mathcal{D}}\ge \PPo{\mathcal{D}\mid \mathcal{H}}-o(1).
                        $$
                    \end{lemma}

                    \begin{proof}
                        Note that 
                        $$
                            \PPo{X\ge t\log\log{t} +1} \le \int_{t\log\log{t}} ^{\infty} \frac{\beta}{x^2} dx = \frac{\beta}{t\log\log{t}}.
                        $$
                        Then by the union bound on the complement of $\mathcal{H}$ we have 
                        \begin{equation*}
                            \PPo{\mathcal{H}}\ge 1-\frac{\beta}{\log\log{t}}=1-o(1).
                        \end{equation*}
                        Therefore 
                        $$
                            \PPo{\mathcal{D}}\ge \PPo{\mathcal{D}\mid \mathcal{H}}\PPo{\mathcal{H}}=\PPo{\mathcal{D}\mid \mathcal{H}}-o(1).
                        $$
                    \end{proof}

                    In order to prove  Theorem~\ref{thm:main2}, letting $\mathcal{D}$ be the event that $r_k(t)=b_k+o(1)$, the previous lemma justifies the use of the truncated variables $Z$'s in the PARID-model. From this point all events and probabilities refer to this model, unless explicitly stated otherwise.

                    Putting together theorems 6.1 and 6.7 in Chung and Lu \cite{MR2283885} we can express the Azuma-Hoeffding inequality in the form below. We will make use of this theorem in the proofs of lemmas \ref{lem:edgeconc} and \ref{lem:marconc1}.

                    \begin{theorem}\label{thm:marconc}
                        Let $(W_\tau)_{\tau = 0}^\infty$, $W_0=0$, be a martingale adapted to the filtration $\mathcal{F} = (\mathcal{F}_\tau)_{\tau = 0}^{\infty}$ satisfying
                        \begin{itemize}
                            \item $\Var{W_\tau\mid \mathcal{F}_{\tau-1}}\le \sigma ^2$ for $1\le \tau \le t$,
                            \item $|W_\tau-W_{\tau-1}|\le b$ for $1\le \tau \le t$,
                        \end{itemize}
                        for some fixed constants $b,\sigma^2$, and for each $1\le t \le \infty$. Then
                        $$
                            \PPo{|W_t|\ge \lambda}\le 2\exp\left(-\frac{\lambda^2}{2 (t \sigma^2+b\lambda/3)}\right).
                        $$
                    \end{theorem}

                    Similarly as before, we consider the number of edges in the graph. Let $\mathcal{C}$ be the event that for $1\le \tau \le t$ the number of edges $L(\tau) = \sum_{i=1}^\tau Z_i$ differs from its expectation by at most $t \log^{2/3}t$.

                    \begin{lemma}\label{lem:edgeconc}
                        We have $\PPo{\mathcal{C}}=1-o(\log^{-2}t)$.
                    \end{lemma}

                    \begin{proof}
                        Let $\bar{Z}(\tau)=\sum_{j=1}^{\tau}\left(Z_j-\Ex{Z}\right)$ for $\tau\ge 1$ and $\bar{Z}(0)=0$. Clearly, $(\bar{Z}(\tau))_{\tau=0}^\infty$ is a martingale with respect to its natural filtration.
                        Now, let $H = \min \{ s \ge 0 \colon |\bar{Z}(s)| > t\log^{2/3}t \}$, and define
                        the stopped martingale $(M(\tau))_{\tau=0}^\infty$ such that $M(\tau) = \bar{Z}(\tau \wedge H)$.
                        Clearly, $M(\tau)$ can change by at most $t\log\log{t}+1=(1+o(1))t\log\log{t}$.
                        By \eqref{eq:secmomZ} we have
                        \begin{equation*}
                            \Var{Z}\le\Ex{Z^2}= 
                            (1+o(1))\beta''t\log\log{t}.
                        \end{equation*}
                        Hence, $\Var{M(\tau)\mid Z_1,\ldots, Z_{\tau-1}}\le\Var{Z}\le(1+o(1))\beta'' t\log\log{t}$. Therefore, by Theorem~\ref{thm:marconc} we have
                        \begin{equation*}
                            \begin{split}
                                \PPo{|M(t)|>t \log^{2/3}t}&\le 2\exp\left(-\frac{t^2\log^{4/3}t}{3 \beta'' t^2 \log\log{t} + t^2 \log^{2/3}t \log\log{t}}\right)\\
                                &\le \exp(-\log^{1/2}t) = o(\log^{-2}t).
                            \end{split}
                        \end{equation*}
                        The last step is justified by recalling that $e^{-x}\le r!/x^r$, for every $x\ge 0$, $r \in \mathbb{N}$. In particular, choose $r=4$.

                        Note that if $|M(\tau)|>t \log^{2/3}t$ for some $\tau\le t$ then, due to the stopping condition, we also have $|M(t)|>t \log^{2/3}t$. Therefore with probability $o(\log^{-2}t)$ for every $1\le \tau \le t$ we have that $L(\tau)$ differs from its expectation by at most $t \log^{2/3}t$.
                    \end{proof}

                    The following asymptotic estimations are necessary for the subsequent results. Let $t_0=t/\log\log{t}$ and for any $1\le s \le t$ let $\mathcal{C}_s$ be the event that for all $1\le \tau \le s$ the number of edges $L(\tau) = \sum_{i=1}^\tau Z_i$ differs from its expectation by at most $t \log^{2/3}t$. Note that $\mathcal{C}_t=\mathcal{C}$.
                    \begin{lemma}\label{treotto}
                        Let $t_0\le s \le t$.
                        Conditional on $\mathcal{C}_{s}$, for every $t_0\le \tau \le s$, we have 
                        \begin{equation}
                            \Ex{L(\tau)}=(1+o(\log^{-1/4}t))\tau \beta'' \log t.
                        \end{equation}
                        and
                        \begin{equation}\label{eq:edgconc}
                            L(\tau)=(1+o(\log^{-1/4}t))\Ex{L(\tau)}=(1+o(\log^{-1/4}t))\tau \beta''\log{t},
                        \end{equation}
                        almost surely.
                    \end{lemma}

                    \begin{proof}
                        First note that
                        \begin{equation}\label{eq:exp}
                            \Ex{Z} = \sum_{i:1\le i < t\log\log t+1} \beta'' i^{-1}= \beta''\int_{1}^{t \log\log{t}} x^{-1} dx +O(1)=(1+o(\log^{-1/4}t)) \beta'' \log{t},
                        \end{equation}
                        and
                        \begin{equation}
                            \Ex{L(\tau)}=(1+o(\log^{-1/4}t))\tau \beta'' \log t.
                        \end{equation}
                        Furthermore, by $\mathcal{C}_{s}$,
                        $L(\tau)<\Ex{L(\tau)}+t \log^{2/3} t$ for every $\tau\le s$.
                        Considering $\Ex{L(\tau)}\ge \Ex{L(t_0)} = (1+o(\log^{-1/4}t))t_0 \beta'' \log(t)$, we have
                        $$
                            \Ex{L(\tau)}+t \log^{2/3} t= (1+O(t (\log^{-1/3}t)/t_0))\Ex{L(\tau)}.
                        $$
                        Hence,
                        $$
                            L(\tau)< (1+O(t (\log^{-1/3}t)/t_0))\Ex{L(\tau)} = (1+o(\log^{-1/4}t))\Ex{L(\tau)}.
                        $$
                        The lower bound follows analogously. 
                    \end{proof}

                    \begin{lemma}\label{lem:Zmoments}
                        For any $t_0\le s \le t$ and $\ell=1,2$,
                        we have
                        $$
                            \Ex{\biggl(\sum_{j=1}^s Z_j\biggr)^{-\ell}}\le \frac{1+o(1)}{(\beta'' s\log {t})^\ell}.
                        $$
                    \end{lemma}

                    \begin{proof}
                        Since $\mathcal{C}\subseteq \mathcal{C}_s$ when $s\le t$ and $\ell\le 2$, by Lemma~\ref{lem:edgeconc} and \eqref{eq:edgconc} in Lemma~\ref{treotto} we have
                        \begin{align*}
                            \Ex{\biggl(\sum_{j=1}^s Z_j\biggr)^{-\ell}}
                            &\le \Ex{\biggl(\sum_{j=1}^s Z_j\biggr)^{-\ell} \mid \mathcal{C}_s}+\Ex{\biggl(\sum_{j=1}^s Z_j\biggr)^{-\ell} \mid \overline{\mathcal{C}}_s}o(\log^{-\ell}t)\\
                            &\le \frac{1+o(1)}{(\beta'' s\log{t})^{\ell}}+s^{-\ell}o(\log^{-\ell}t)=\frac{1+o(1)}{(\beta''s\log{t})^{\ell}},
                        \end{align*}
                        where in the penultimate step we use that $Z_j\ge 1$ for every $j\in \mathbb{N}$.
                    \end{proof}

                    By means of the following two lemmas we will establish the value of $\Ex{R_k(t)}$. While we could use the aforementioned results of Bhamidi \cite{bhamidi2007universal} to deduce that $b_k=\Ex{\ratdeg_k(t)}+o(1)$ for completeness we include an alternate proof here.
                    More precisely we show that $\Ex{\numdeg_k(t)}$ is close to $(t+1) b_k$. As a first step we provide an iterative method for calculating $\Ex{\numdeg_k(t)}$.

                    \begin{lemma}\label{lem:expchange}
                        For any $k\le \log\log t$ and $t_0\le \tau \le t$ we have
                        $$
                            \Ex{\numdeg_{k}(\tau+1)}=\Ex{\numdeg_{k}(\tau)}+\frac{k-1}{2}\frac{\Ex{\numdeg_{k-1}(\tau)}}{\tau}-\frac{k}{2}\frac{\Ex{\numdeg_{k}(\tau)}}{\tau}+\PPo{Z=k}+o((\log\log t)^{-1}).
                        $$
                    \end{lemma}

                    \begin{proof}
                        Denote by $\numdeg_j^{+i}(\tau+1)$ the number of vertices which had degree $j$ in step $\tau$ and have degree $j+i$ in step $\tau+1$, $i =1,\dots, \lceil \log\log {t}\rceil$, and let $H(\tau+1)$ be the number of vertices, with degree at most $k$ whose degree changes by at least 2 in step $\tau+1$.
                        Then
                        \begin{align}\label{eq:exbreakup}
                            & \Ex{\numdeg_{k}(\tau+1)}
                            =\Ex{\numdeg_{k}(\tau)}
                            - \sum_{i=1}^{\lceil \log\log {t}\rceil } \Ex{\numdeg_{k}^{+i}(\tau+1)} + \sum_{i=1}^{k-1}\Ex{\numdeg_{k-i}^{+i}(\tau+1)} +\PPo{Z=k}\notag \\
                            & = \Ex{\numdeg_{k}(\tau)}+\Ex{\numdeg_{k-1}^{+1}(\tau+1)}-\Ex{\numdeg_k^{+1}(\tau+1)}+\PPo{Z=k}+O(\Ex{H(\tau+1)}).    
                        \end{align}

                        Recall that $\mathcal{C}_\tau$ is the event that for $1\le s \le \tau$ the number of edges $L(s) = \sum_{i=1}^s Z_i$ differs from its expectation by at most $t \log^{2/3}t$. Then for any $j\le k$
                        \begin{align}
                            \Ex{\numdeg_j^{+1}(\tau+1)}
                            &=\Ex{\Ex{\numdeg_j^{+1}(\tau+1)\mid \numdeg_{j}(\tau),L(\tau),Z_{\tau+1}}}\nonumber\\
                            &=\Ex{\numdeg_{j}(\tau)Z_{\tau+1} \frac{j}{2L(\tau)}\left(1-\frac{j}{2L(\tau)}\right)^{Z_{\tau+1}-1}}\nonumber\\
                            &=\Ex{1_{\mathcal{C}_{\tau}} \numdeg_{j}(\tau) \frac{j Z_{\tau+1}}{2L(\tau)}\left(1-\frac{j}{2L(\tau)}\right)^{Z_{\tau+1}-1}}\nonumber\\
                            &\quad +\Ex{1_{\overline{\mathcal{C}}_{\tau}}\numdeg_{j}(\tau) \frac{j Z_{\tau+1}}{2L(\tau)}\left(1-\frac{j}{2L(\tau)}\right)^{Z_{\tau+1}-1}}.\label{eq:expindsplit}
                        \end{align}
                        Using  \eqref{eq:edgconc} of Lemma \ref{treotto}, and $j Z_{\tau+1}\le t (\log\log{t})^2+\log\log{t}=o(\Ex{L(\tau)})$ we have 
                        \begin{align}
                            &\Ex{1_{\mathcal{C}_{\tau}} \numdeg_{j}(\tau) \frac{jZ_{\tau+1}}{2L(\tau)}\left(1-\frac{j}{2L(\tau)}\right)^{Z_{\tau+1}-1}}\nonumber\\
                            &=\Ex{1_{\mathcal{C}_{\tau}} \numdeg_{j}(\tau) \frac{jZ_{\tau+1}}{(2+o(\log^{-1/4}t))\Ex{L(\tau)}}\left(1-O\left(\frac{\log\log{t}\, Z_{\tau+1}}{\Ex{L(\tau)}}\right)\right)}\nonumber\\
                            &=\frac{j}{2}\frac{\Ex{1_{\mathcal{C}_{\tau}} \numdeg_{j}(\tau)}}{\Ex{L(\tau)}} \Ex{Z}+o\left(\tau \frac{ \log\log t\, \Ex{Z}}{\log^{1/4}t \Ex{L(\tau)}}\right)
                            -O\left(\tau\frac{(\log\log{t})^2 \Ex{Z^2}}{\Ex{L(\tau)}^2}\right),\label{eq:expindsplit1}
                        \end{align}
                        where in the last step we used that $R_{j}(\tau)\le \tau$ almost surely, that $Z_{\tau+1}$ is independent of both $\mathcal{C}_{\tau}$ and $R_j(\tau)$, and that $1/(1+o(1)) = 1+o(1)$.
                        Recall that $\Ex{L(\tau)}=\tau \Ex{Z}$ implying
                        $$
                            o\left(\tau \frac{ \log\log t\, \Ex{Z}}{\log^{1/4}t \Ex{L(\tau)}}\right)=o((\log\log t)^{-1}).
                        $$ 
                        Also recall that $\tau\ge t_0= t/\log\log t$, which together with \eqref{eq:firstmomZ}, \eqref{eq:secmomZ} and $\Ex{L(\tau)}=\tau \Ex{Z}$ implies
                        $$
                            O\left(\tau\frac{(\log\log{t})^2 \Ex{Z^2}}{\Ex{L(\tau)}^2}\right)=o((\log\log t)^{-1}).
                        $$
                        In addition by Lemma~\ref{lem:edgeconc} and $\numdeg_j(\tau)\le \tau$ we have
                        $$
                            \Ex{R_j(\tau)}=\Ex{R_j(\tau)1_{\mathcal{C}_{\tau}}}+\Ex{R_j(\tau)1_{\overline{\mathcal{C}}_{\tau}}}=\Ex{R_j(\tau)1_{\mathcal{C}_{\tau}}}+o\left(\frac{\tau}{\log^2t}\right)
                        $$
                        and thus
                        \begin{align}
                            \frac{j}{2}\frac{\Ex{1_{\mathcal{C}_{\tau}} \numdeg_{j}(\tau)}}{\Ex{L(\tau)}} \Ex{Z}&+o\left(\tau \frac{ \log\log t\, \Ex{Z}}{\log^{1/4}t \Ex{L(\tau)}}\right)
                            -O\left(\tau\frac{(\log\log{t})^2 \Ex{Z^2}}{\Ex{L(\tau)}^2}\right)\nonumber \\
                            &=\frac{j}{2}\frac{\Ex{\numdeg_{j}(\tau)}}{\tau}+o((\log\log t)^{-1}).
                            \label{eq:expindsplit2}
                        \end{align}
                        As for the other term
                        \begin{align}\label{eq:expindsplit3}
                            \Ex{1_{\overline{\mathcal{C}}_{\tau}}\numdeg_{j}(\tau) \frac{j Z_{\tau+1}}{2L(\tau)}\left(1-\frac{j}{2L(\tau)}\right)^{Z_{\tau+1}}}
                            &=O\left(\Ex{1_{\overline{\mathcal{C}}_{\tau}}\tau \frac{Z_{\tau+1}}{\tau}}\log\log t\right)\notag \\
                            &=o\left(\frac{\Ex{Z}}{\log^2 t}\log\log t\right)=o((\log\log t)^{-1}).
                        \end{align}
                        Using \eqref{eq:expindsplit1},\eqref{eq:expindsplit2} and \eqref{eq:expindsplit3} to evaluate \eqref{eq:expindsplit}, we have
                        \begin{equation}\label{eq:onestepexp}
                            \Ex{\numdeg_j^{+1}(\tau+1)}=\frac{j}{2}\frac{\Ex{\numdeg_{j}(\tau)}}{\tau}+o((\log\log t)^{-1}).
                        \end{equation}
                        Finally we consider $\Ex{H(\tau+1)}$. Similarly to the computation of $\Ex{\numdeg_j^{+1}(\tau+1)}$,
                        \begin{align}
                            \Ex{H(\tau+1)}&=\Ex{1_{\mathcal{C}_\tau} H(\tau+1)}+\Ex{1_{\overline{\mathcal{C}}_\tau} H(\tau+1)}\nonumber\\
                            &\le \frac{(1+o(1))\tau}{4\Ex{L(\tau)}^2}\Ex{k^2 Z_{\tau+1}^2}+\frac{\tau}{\tau^2}\Ex{k^2 Z_{\tau+1}^2}o(\log^{-2}t)=o((\log\log t)^{-1}). \label{eq:morestepexp}
                        \end{align}
                        The result follows by applying \eqref{eq:onestepexp} for $j=k-1,k$ and \eqref{eq:morestepexp} in \eqref{eq:exbreakup}.
                    \end{proof}

                    Using the previously established iteration, in the following lemma we establish the expected number of vertices with degree $k$.

                    \begin{lemma}\label{lem:expR}
                        For every $k\ge 1$ we have	
                        $$
                            \Ex{\numdeg_k(t)}=(t+1)b_k+o(t).
                        $$
                    \end{lemma}

                    \begin{proof}
                        For $k\ge 1$ define
                        \begin{align}\label{newlabel}
                            b_k':=b_k'(t)= \frac{2 \beta''}{k(k+1)(k+2) }\left(\sum_{i=1}^{\min\{k, \lceil t \log\log{t} \rceil\}}\left(1+\frac{1}{i}\right)\right),
                        \end{align}
                        and set $b_0'=0$.
                        Note that this satisfies the equation
                        \begin{equation}\label{eq:bkconn}
                            b_k'=\frac{k-1}{2}b_{k-1}'-\frac{k}{2}b_k'+\PPo{Z=k}
                        \end{equation}
                        for all $k\ge 1$. This can be easily checked by recalling that $\PPo{Z=k}=0$ for all $k \ge t \log \log t + 1$ and substituting \eqref{newlabel} into \eqref{eq:bkconn}. Thus
                        \begin{equation}\label{mototomo}
                            \sum_{k=1}^{\infty}b_k'=\sum_{k=1}^\infty \PPo{Z=k}=1.
                        \end{equation}

                        Now assume $k\le \log\log t$. For $1\le \tau \le t$ let 
                        $$
                            e_\tau=\max_{k\le \log\log{t}}|\Ex{\numdeg_k(\tau)}-(\tau+1) b_k'|.
                        $$
                        We  show that for every $t_0 \le \tau \le t$ we have
                        \begin{align}\label{4a}
                            e_\tau\le t_0+1+o((\tau-t_0)/\log\log{t}).
                        \end{align}
                        The proof is by induction. 
                        If $\tau=t_0$, we have
                        $$
                            |\Ex{\numdeg_k(t_0)}-(t_0+1) b_k'|\le t_0+1.
                        $$

                        By Lemma~\ref{lem:expchange} for $t_0< \tau \le t$ 
                        \begin{align*}
                            \Ex{\numdeg_k(\tau+1)}
                            &=\left(1-\frac{k}{2\tau}\right)\Ex{\numdeg_{k}(\tau)}+\frac{k-1}{2}\frac{\Ex{\numdeg_{k-1}(\tau)}}{\tau}+\PPo{Z=k}+o((\log\log{t})^{-1})\\
                            &=\left(1-\frac{k}{2\tau}\right)\left((\tau+1) b_k'+\Ex{\numdeg_{k}(\tau)}-(\tau+1) b_k'\right)\\
                            &\quad+\frac{k-1}{2}\frac{(\tau+1) b_{k-1}'+\Ex{\numdeg_{k-1}(\tau)}-(\tau+1) b_{k-1}'}{\tau}+\PPo{Z=k}+o((\log\log{t})^{-1})\\
                            &\stackrel{\eqref{eq:bkconn}}{=}(\tau+2) b_k'+ \left(1-\frac{k}{2\tau}\right)\left(\Ex{\numdeg_{k}(\tau)}-(\tau+1) b_k'\right)\\
                            &\quad+\frac{k-1}{2}\frac{\Ex{\numdeg_{k-1}(\tau)}-(\tau+1) b_{k-1}'}{\tau}+o((\log\log{t})^{-1}),
                        \end{align*}
                        where in the last step some terms have been included in the $o((\log\log t)^{-1})$ term.
                        We have 
                        \begin{align}\label{aaaa}
                            |\Ex{\numdeg_k(\tau+1)}-(\tau+2) b_k'|&\le \left(1-\frac{k}{2\tau}\right)\left|\Ex{\numdeg_{k}(\tau)}-(\tau+1) b_k'\right|\notag \\
                            &+\frac{k-1}{2\tau}\left|\Ex{\numdeg_{k-1}(\tau)}-(\tau+1) b_{k-1}'\right|+o((\log\log{t})^{-1})\notag\\
                            &\le \left(1-\frac{k}{2 \tau}\right)e_{\tau}+\frac{k-1}{2\tau}e_\tau+o((\log\log{t})^{-1})\notag\\
                            &\le e_{\tau}+o((\log\log{t})^{-1}).
                        \end{align}
                        Hence, formula \eqref{4a} follows by using the inductive hypothesis.

                        Since $t_0=t/\log\log{t}$ we have 
                        \begin{equation}\label{eq:expnumdeg}
                            \Ex{\numdeg_k(t)}=(t+1) b_k'+O(t/\log\log{t})+o((t-t_0)/\log\log{t})=(t+1) b_k'+O(t/\log\log{t}).
                        \end{equation}
                        Recalling $b_k'=b_k+o(1)$, we deduce the result for $k\le \log\log {t}$.

                        On the other hand since $\sum_{k=\lceil(\log\log t)^{1/2}\rceil+1}^\infty b_k'=o(1)$ we have
                        \begin{align*}
                            \sum_{k=1}^{\lceil(\log\log t)^{1/2}\rceil} \Ex{\numdeg_k(t)} & \overset{\eqref{eq:expnumdeg}}{=}
                            \sum_{k=1}^{\lceil(\log\log t)^{1/2}\rceil} (t+1) b_k'+o(t) \\
                            & \overset{\eqref{mototomo}}{=} t+1 -(t+1)\sum_{k=\lceil(\log\log t)^{1/2}\rceil+1}^\infty b_k' + o(t) \\
                            & = t+1 -o(t).
                        \end{align*}
                        By considering that $\sum_{j=1}^\infty \Ex{\numdeg_j(t)}= \Ex{\sum_{j=1}^\infty\numdeg_j(t)}=t+1$, it follows that $\Ex{\numdeg_k(t)}=o(t)$ for every $k\ge (\log\log{t})^{1/2}$. The result follows as in this range we also have $b_k=o(1)$.
                    \end{proof}

                    For the remainder of the paper we will show that $\numdeg_{k}(t)$ is concentrated around its expectation for every $k$. We will use a martingale argument to achieve this. First we will show that $\Ex{\numdeg_k(t)|Z_1,\ldots,Z_t}$ is concentrated around its expectation. Then that, conditionally on $\{Z_1=y_1, \ldots, Z_t=y_t\}$, if the total number of edges added is close to its expectation, then the degree sequence of the PARID-model is also concentrated around its expectation. We start with a technical lemma on the difference of two products.

                    \begin{lemma}\label{lem:proddiff}
                        Let $\xi_1,\ldots,\xi_n$ and $\zeta_1,\ldots,\zeta_n$ be positive real numbers satisfying $\prod_{h=1}^n \xi_h\ge \prod_{h=1}^n \zeta_h$. Then
                        $$
                            \prod_{h=1}^n \xi_h- \prod_{h=1}^n \zeta_h \le \prod_{h=1}^n \xi_h \sum_{j=1}^n \left| 1-\frac{\zeta_j}{\xi_j}\right|.
                        $$ 
                    \end{lemma}

                    \begin{proof}
                        Note that $\prod_{h=1}^n \zeta_h/\xi_h\le 1$ and without loss of generality in the framework of our proof assume that $\zeta_1/\xi_1\le \zeta_2/\xi_2\le \ldots \le \zeta_n/\xi_n$. Then for every $j=1,\ldots,n$ we have
                        $$
                            \prod_{h=1}^j \frac{\zeta_h}{\xi_h}\le 1,
                        $$
                        as either $\zeta_{j+1}/\xi_{j+1}>1$ implying that $\zeta_h/\xi_h> 1$ for every $h\ge j+1$ and thus
                        $$
                            \prod_{h=1}^j \frac{\zeta_h}{\xi_h}\le \prod_{h=1}^{n} \frac{\zeta_h}{\xi_h} \le 1,
                        $$
                        or $\zeta_{j+1}/\xi_{j+1}\le 1$, implying that $\zeta_h/\xi_h\le 1$ for every $h\le j$. Therefore
                        $$
                            1-\prod_{h=1}^n\frac{\zeta_h}{\xi_h}=\sum_{j=1}^{n} \left(\prod_{h=1}^{j-1}\frac{\zeta_h}{\xi_h}-\prod_{h=1}^{j}\frac{\zeta_h}{\xi_h}\right)\le \sum_{j=1}^{n} \left|1-\frac{\zeta_j}{\xi_j}\right| \left(\prod_{h=1}^{j-1}\frac{\zeta_h}{\xi_h}\right)\le \sum_{j=1}^{n} \left|1-\frac{\zeta_j}{\xi_j}\right|.
                        $$
                        The result follows by multiplying both sides of the inequality by $\prod_{h=1}^n \xi_h$.
                    \end{proof}

                    As we will use a martingale argument to show that the random variable $\Ex{\numdeg_k(t)\mid Z_1,\ldots,Z_t}$ is concentrated we need to analyse how changing the value of one of the $Z_i$-s affects this conditional expectation, in particular the probability that one of the vertices has degree $k$.

                    \begin{lemma}\label{lem:condchange}
                        Let $1\le \tau \le t$.
                        For $y_1,\ldots,y_t$ and $y'$ positive integers, denote the events $\{Z_1=y_1,\ldots,Z_\tau=y_\tau,\ldots,Z_t=y_t\}$ and $\{Z_1=y_1,\ldots,Z_\tau=y',\ldots,Z_t=y_t\}$ with $\mathcal{Y}$ and $\mathcal{Y}'$ respectively. Then for any positive integer $k$ and any positive integer $k< i\le t$, $i\neq \tau$  we have
                        $$
                            \left|\PPo{d_{i}(t)=k \mid \mathcal{Y}}-\PPo{d_{i}(t)=k \mid \mathcal{Y}'}\right|\le \frac{3k |y_\tau-y'|}{\sum_{\substack{j=1\\j\neq\tau}}^{i}y_j}\left(1+\frac{y'+\sum_{j=\tau}^t y_j}{\sum_{\substack{j=1\\j\neq\tau}}^{i}y_j}\right).
                        $$
                    \end{lemma}

                    \begin{proof}
                        Since $d_{i}(t)$ is a monotone non decreasing function of $t$ both of these probabilities are zero, and hence the statement holds, if  the initial degree $y_i$ of $v_i$ is strictly larger than $k$. 
	
                        Now consider the case $y_i\le k$ and without loss of generality assume $y_\tau\ge y'$. 
                        In order to have a clearer view of the process we will break up each step into rounds, see Figure~\ref{fig:rounds}. In each round exactly one edge is inserted, and thus the $j$-th step consists of $Z_j$ rounds. Denote by $(j,\ell)$ the $\ell$-th round in step $j$.
                        Recall that 
                        $L(\tau)=\sum_{j=1}^{\tau}Z_{j}$
                        and denote by $L_{\mathcal{Y}}(\tau)$ and $L_{\mathcal{Y}'}(\tau)$ the value of this random variable when conditioning on $\mathcal{Y}$ and $\mathcal{Y}'$ respectively.
	
                        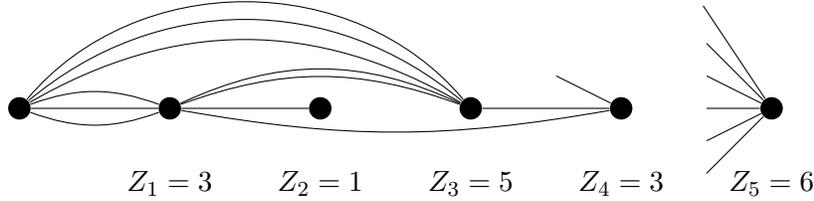
\begin{figure}[H]
                            \centering
                            \begin{tikzpicture}
	
                                \node (X0) at (0,0) [circle,fill=black, inner sep=3pt] {};
                                \node (X1) at (2,0) [circle, fill=black, inner sep=3pt] {};
                                \node (X2) at (4,0) [circle, fill=black, inner sep=3pt] {};
                                \node (X3) at (6,0) [circle, fill=black,inner sep=3pt] {};
                                \node (X4) at (8,0) [circle, fill=black, inner sep=3pt] {};
                                \node (X5) at (10,0) [circle, fill=black, inner sep=3pt] {};
    
                                \node (Y43) at (7,0.5) {};
                                \node (Y51) at (9,-1) {};
                                \node (Y52) at (9,-0.5) {};
                                \node (Y53) at (9,0) {};
                                \node (Y54) at (9,0.5) {};
                                \node (Y55) at (9,1) {};
                                \node (Y56) at (9,1.5) {};    
	
                                \node (Z1) at (2,-1) {$Z_1=3$};
                                \node (Z2) at (4,-1) {$Z_2=1$};
                                \node (Z3) at (6,-1) {$Z_3=5$};
                                \node (Z4) at (8,-1) {$Z_4=3$};
                                \node (Z5) at (10,-1) {$Z_5=6$};
	
                                \draw (X0)--(X1);
                                \draw (X0) to[out=-20,in=-160] (X1);
                                \draw (X0) to[out=20,in=160] (X1);
                                \draw (X1)--(X2);
                                \draw (X0) to[out=30,in=150] (X3);
                                \draw (X0) to[out=40,in=140] (X3);
                                \draw (X0) to[out=50,in=130] (X3);
                                \draw (X1) to[out=20,in=160] (X3);
                                \draw (X1) to[out=25,in=155] (X3);
                                \draw (X3)--(X4);
                                \draw (X1) to[out=-10,in=-170] (X4);
                                \draw (X4)--(Y43);
                                \draw (X5)--(Y51);
                                \draw (X5)--(Y52);
                                \draw (X5)--(Y53);
                                \draw (X5)--(Y54);
                                \draw (X5)--(Y55);
                                \draw (X5)--(Y56);

                            \end{tikzpicture}\\
                            \caption{The process at the beginning of round (4,3), conditional on the $Z_j$-s.}
                            \label{fig:rounds}
                        \end{figure}

                        In order for the vertex $v_i$ to have degree $k$ by the end of step $t$, the number of rounds where an edge is connected to $v_i$ between the end of step $i$ and the end of step $t$ must be exactly $k-y_i$.
                        Let $\mathcal{R}_{\mathcal{Y}}(i)$ be the set of rounds between the end of step $i$ and the end of step $t$, conditional on $\mathcal{Y}$, that is 
                        $$
                            \mathcal{R}_{\mathcal{Y}}(i)=\{(j,\ell):i<j\le t, 1\le \ell \le y_j\},
                        $$
                        and define $\mathcal{R}_{\mathcal{Y}'}(i)$ analogously. Note that $\mathcal{R}_{\mathcal{Y}'}(i)\subseteq \mathcal{R}_{\mathcal{Y}}(i)$, indeed the two sets can only differ with respect to the rounds of the form $(\tau,\ell)$ where $y'<\ell \le y_{\tau}$. Recall that only rounds in the steps larger than $i$ are included in $\mathcal{R}_{\mathcal{Y}'}(i)$ and $\mathcal{R}_{\mathcal{Y}}(i)$. Therefore if $\tau< i$ then $\mathcal{R}_{\mathcal{Y}'}(i)= \mathcal{R}_{\mathcal{Y}}(i)$, otherwise when $\tau>i$ then
                        \begin{equation}\label{eq:diffR}
                            \mathcal{R}_{\mathcal{Y}}(i)\setminus \mathcal{R}_{\mathcal{Y}'}(i)= \{(\tau,\ell): y'<\ell \le y_{\tau}\}.
                        \end{equation}

                        For a set $S\subset \mathbb{N}\times \mathbb{N}$ let $\mathcal{E}_S$ be the event that the degree of $v_i$ increases exactly in the rounds corresponding to $S$. Then
                        \begin{align*}
                            \PPo{d_{i}(t)=k \mid \mathcal{Y}}&=\sum_{\substack{S\subseteq \mathcal{R}_{\mathcal{Y}}(i)\\|S|=k-y_i}}\PPo{\mathcal{E}_S\mid \mathcal{Y}}\\
                            \PPo{d_{i}(t)=k \mid \mathcal{Y}'}&=\sum_{\substack{S'\subseteq \mathcal{R}_{\mathcal{Y'}}(i)\\|S'|=k-y_i}}\PPo{\mathcal{E}_{S'}\mid \mathcal{Y}'}.
                        \end{align*}
	
                        Our aim is to establish the difference of these two sums, and we will compare the individual terms within the sums. 
                        More precisely we will compare the terms on which $S$ and $S'$ coincide, and account for the terms corresponding to sets involving elements from $\mathcal{R}_{\mathcal{Y}}(i)\setminus \mathcal{R}_{\mathcal{Y}'}(i)$ separately.
                        Formally
                        \begin{align}
                            \left|\PPo{d_{i}(t)=k \mid \mathcal{Y}}-\PPo{d_{i}(t)=k \mid \mathcal{Y}'}\right|&\le \sum_{\substack{S\subseteq \mathcal{R}_{\mathcal{Y}'}(i)\\|S|=k-y_i}}\left|\PPo{\mathcal{E}_{S}\mid \mathcal{Y}}-\PPo{\mathcal{E}_S\mid \mathcal{Y}'}\right|\nonumber\\
                            &+\sum_{\ell=y'+1}^{y_{\tau}}\PPo{\mathcal{F}_\ell\mid \mathcal{Y}}, \label{eq:probdiff}
                        \end{align}
                        where $\mathcal{F}_\ell$ denotes the event that the degree of $v_i$ increases in round $(\tau,\ell)$, and has degree at most $k-1$ by the end of step $\tau-1$, that is $d_i(\tau-1)< k$. Note that in \eqref{eq:probdiff},  the probability is increased as we require at least one round where an edge is connected to $v_i$ instead of exactly $k-y_i$.
                        Recall that $i\neq\tau$ then for $y'< \ell \le y_{\tau}$ we have
                        $$
                            \PPo{\mathcal{F}_\ell\mid \mathcal{Y}}\le
                            \left\{
                            \begin{array}{ll}
                                \frac{k}{2L_{{\mathcal{Y}}}(\tau-1)}  & \mbox{if } i < \tau ,\\
                                0 & \mbox{if } i > \tau,
                            \end{array}
                            \right.
                        $$
                        implying
                        \begin{equation}\label{eq:error1}
                            \sum_{\ell=y'+1}^{y_{\tau}}\PPo{\mathcal{F}_\ell\mid \mathcal{Y}}\le |y_\tau-y'|\frac{k}{2L_{{\mathcal{Y}}}(i)}.
                        \end{equation}
                        In \eqref{eq:probdiff} we consider $S \subseteq \mathcal{R}_{\mathcal{Y}'}(i)$ and note that $d_i(j)$ remains the same for every $i\le j$ if we condition on the events $\mathcal{E}_S$ and $\mathcal{Y}$ or on the events $\mathcal{E}_S$ and $\mathcal{Y}'$. Denote this value by $d_{i,S}(j)$.  
                        Then we have
                        \begin{align*}
                            \PPo{\mathcal{E}_S\mid \mathcal{Y}'}&=\prod_{(j,\ell)\in S} \frac{d_{i,S}(j-1)}{2L_{\mathcal{Y}'}(j-1)} \prod_{(j,\ell)\in \mathcal{R}_{\mathcal{Y}'}(i)\setminus S}\left(1-\frac{d_{i,S}(j-1)}{2L_{\mathcal{Y}'}(j-1)}\right)\\
                            \PPo{\mathcal{E}_{S}\mid \mathcal{Y}}&=\prod_{(j,\ell)\in S} \frac{d_{i,S}(j-1)}{2L_{\mathcal{Y}}(j-1)} \prod_{(j,\ell)\in \mathcal{R}_{\mathcal{Y}}(i)\setminus S}\left(1-\frac{d_{i,S}(j-1)}{2L_{\mathcal{Y}}(j-1)}\right).
                        \end{align*}
 
                        In order to simplify the comparison of these two probabilities we introduce the following value
                        $$
                            A_S=\prod_{(j,\ell)\in S} \frac{d_{i,S}(j-1)}{2L_{\mathcal{Y}}(j-1)} \prod_{(j,\ell)\in\mathcal{R}_{\mathcal{Y}'}(i)\setminus S}\left(1-\frac{d_{i,S}(j-1)}{2L_{\mathcal{Y}}(j-1)}\right).
                        $$
                        Note that the only difference between $A_S$ and $\PPo{\mathcal{E}_{S}\mid \mathcal{Y}}$ is that the second product is over $\mathcal{R}_{\mathcal{Y}'}(i)\setminus S$ instead of $\mathcal{R}_{\mathcal{Y}}(i)\setminus S$. 
                        Recall that if $\tau<i$ then $\mathcal{R}_{\mathcal{Y}'}(i)=\mathcal{R}_{\mathcal{Y}}(i)$, implying $\PPo{\mathcal{E}_{S}\mid \mathcal{Y}}=A_S$. On the other hand for $i<\tau$ by \eqref{eq:diffR} we have
                        \begin{equation}\label{eq:error2}
                            |\PPo{\mathcal{E}_{S}\mid \mathcal{Y}}-A_S|= A_S\left(1-\prod_{\ell=y'+1}^{y_{\tau}}\left(1-\frac{d_{i,S}(\tau-1)}{2L_{\mathcal{Y}}(\tau-1)}\right)\right)\le  A_S\frac{k|y_\tau-y'|}{2L_{\mathcal{Y}}(i)},
                        \end{equation}
                        where, for the last inequality, we applied Lemma~\ref{lem:proddiff} choosing $\xi_h = 1$ and $\zeta_h = 1-\frac{d_{i,S}(\tau -1)}{2L_{\mathcal{Y}}(\tau -1)}$. Further, we took into account that $d_{i,S}(\tau-1)\le k$  and $L_{\mathcal{Y}}(\tau-1)\ge L_{\mathcal{Y}}(i)$. 
	
                        Next, we focus on $|\PPo{\mathcal{E}_{S}\mid \mathcal{Y}'}-A_S|$.
                        As the only difference between $\mathcal{Y}$ and $\mathcal{Y}'$ is that the value of $Z_{\tau}$ changes from $y_{\tau}$ to $y'$, for every $1\le j \le t$ we have
                        \begin{equation}\label{renumer}
                            L_{\mathcal{Y}}(j)-L_{\mathcal{Y}'}(j) = \left\{
                            \begin{array}{ll}
                                0  & \mbox{if } j< \tau \\
                                y_\tau-y' & \mbox{if } j\ge \tau.
                            \end{array}
                            \right. 
                        \end{equation}
                        By \eqref{renumer} we have, for every $1\le j < \tau$, that
                        \begin{equation}\label{eq:diff0}
                            \left| 1-\frac{d_{i,S}(j)}{2 L_{\mathcal{Y}}(j)} \frac{2 L_{\mathcal{Y}'}(j)}{d_{i,S}(j)} \right|=0 \quad \mbox{and}\quad \left|1-\left(1-\frac{d_{i,S}(j)}{2L_{\mathcal{Y}'}(j)}\right)\left(1-\frac{d_{i,S}(j)}{2L_{\mathcal{Y}}(j)}\right)^{-1}\right|=0.
                        \end{equation}
                        In the following we will also establish these values when $\tau\ge i$. Using the inequality of arithmetic and geometric means one can quickly show that for any $0<a\le b$ we have $|1-a/b|\le |1-b/a|$.
                        Therefore for any $\max\{i,\tau\}\le j \le t-1$ and $S$ we have
                        \begin{equation}\label{eq:diff1}
                            \left| 1-\frac{d_{i,S}(j)}{2 L_{\mathcal{Y}}(j)} \frac{2 L_{\mathcal{Y}'}(j)}{d_{i,S}(j)} \right|\le \left| 1-\frac{d_{i,S}(j)}{2 L_{\mathcal{Y}'}(j)} \frac{2 L_{\mathcal{Y}}(j)}{d_{i,S}(j)} \right| \le \frac{2|y_{\tau}-y'|}{2 L_{\mathcal{Y}'}(j)},
                        \end{equation}
                        and also
                        \begin{align}
                            \left|1-\left(1-\frac{d_{i,S}(j)}{2L_{\mathcal{Y}'}(j)}\right)\left(1-\frac{d_{i,S}(j)}{2L_{\mathcal{Y}}(j)}\right)^{-1}\right|
                            &\le \left|1-\left(1-\frac{d_{i,S}(j)}{2L_{\mathcal{Y}}(j)}\right)\left(1-\frac{d_{i,S}(j)}{2L_{\mathcal{Y}'}(j)}\right)^{-1}\right|\nonumber\\
                            &=\left|\left(\frac{d_{i,S}(j)}{2L_{\mathcal{Y}}(j)}-\frac{d_{i,S}(j)}{2L_{\mathcal{Y}'}(j)}\right)\left(1-\frac{d_{i,S}(j)}{2L_{\mathcal{Y}'}(j)}\right)^{-1}\right|\nonumber\\
                            &\le \frac{2k|y_\tau-y'|}{2L_{\mathcal{Y}}(j)(2L_{\mathcal{Y}'}(j)-k)}\le \frac{k|y_\tau-y'|}{(L_{\mathcal{Y}'}(j))^2},\label{eq:diff2}
                        \end{align}
                        where in the second-to-last step we used that $d_{i,S}(j) \le k$ and  \eqref{renumer}. Further, in the last step we also used that $k<i \le j$ which implies $L_{\mathcal{Y}'}(j)\ge L_{\mathcal{Y}'}(k)\ge k$ and $L_{\mathcal{Y}'}(j)\le L_{\mathcal{Y}}(j)$.
                        Applying twice Lemma~\ref{lem:proddiff}, first on $A_S - \PPo{\mathcal{E}_S\mid \mathcal{Y}'}$ and then on $\PPo{\mathcal{E}_S\mid \mathcal{Y}'}-A_S$, and subsequently using \eqref{eq:diff0} for the steps $i\le j < \tau$, followed by \eqref{eq:diff1} and \eqref{eq:diff2} for the steps $\max\{i,\tau \} \le j \le t$, we obtain
                        \begin{align}
                            |\PPo{\mathcal{E}_S\mid \mathcal{Y}'}-A_{S}|
                            &\le A_{S}\left(\sum_{\substack{j\in S\\j\ge \tau}}\frac{ |y_\tau-y'|}{L_{\mathcal{Y}'}(i)}+\sum_{\substack{(j,\ell)\in \mathcal{R}_{\mathcal{Y}'}(i)\\j\ge \tau}}\frac{k|y_\tau-y'|}{(L_{\mathcal{Y}'}(i))^2}
                            \right)\nonumber \\
                            &+\PPo{\mathcal{E}_S\mid \mathcal{Y}'}\left(\sum_{\substack{j\in S\\ j\ge \tau}}\frac{ |y_\tau-y'|}{L_{\mathcal{Y}'}(i)}+\sum_{\substack{(j,\ell)\in \mathcal{R}_{\mathcal{Y}'}(i)\\j\ge \tau}}\frac{k|y_\tau-y'|}{(L_{\mathcal{Y}'}(i))^2}\right)\nonumber\\
                            &\le(A_S+\PPo{\mathcal{E}_S\mid \mathcal{Y}'})k|y_{\tau}-y'|\left(\frac{1}{L_{\mathcal{Y}'}(i)}+\frac{\sum_{j=\tau}^ty_j}{(L_{\mathcal{Y}'}(i))^2}\right),
                            \label{eq:error3}
                        \end{align}
                        where in the last line we used that $|\{(j,\ell)\in \mathcal{R}_{\mathcal{Y}'}(i):j\ge \tau\}|\le \sum_{j=\tau}^t y_j$.
                        Note that
                        $$
                            \sum_{\substack{S\subseteq \mathcal{R}_{\mathcal{Y}'}(i)\\|S|=k-y_i}}\PPo{\mathcal{E}_S\mid \mathcal{Y}'}\le1 \quad \mbox{and} \quad \sum_{\substack{S\subseteq \mathcal{R}_{\mathcal{Y}'}(i)\\|S|=k-y_i}}A_{S} \le 1,
                        $$
                        as in both sums we consider the probabilities of disjoint events specifying the exact rounds in $\mathcal{R}_{\mathcal{Y}'}(i)$ where the degree of $v_i$ increases, conditional on $\mathcal{Y}'$ and on both $\mathcal{Y}'$ and $\mathcal{Y}$ respectively.
                        Hence,
                        $$
                            \sum_{\substack{S\subseteq \mathcal{R}_{\mathcal{Y}'}(i)\\|S|=k-y_i}}(A_S + \PPo{\mathcal{E}_S\mid \mathcal{Y}'})\le 2,
                        $$
                        and it follows from \eqref{eq:error2} and \eqref{eq:error3}
                        \begin{equation}\label{eq:error4}
                            \sum_{(j,\ell)\in \mathcal{R}_{\mathcal{Y}'}(i)}|\PPo{\mathcal{E}_S\mid \mathcal{Y}}-\PPo{\mathcal{E}_S\mid \mathcal{Y}'}|
                            \le k|y_\tau-y'| \left(\frac{5}{2L_{\mathcal{Y}'}(i)}+\frac{2\sum_{j=\tau}^ty_j}{(L_{\mathcal{Y}'}(i))^2}\right).
                        \end{equation}
    
                        Substituting \eqref{eq:error1} and \eqref{eq:error4} into \eqref{eq:probdiff} gives
                        \begin{equation*}
                            \left|\PPo{d_{i}(t)=k \mid \mathcal{Y}}-\PPo{d_{i}(t)=k \mid \mathcal{Y}'}\right|
                            \le k|y_\tau-y'| \left(\frac{3}{L_{\mathcal{Y}'}(i)}+\frac{2\sum_{j=\tau}^ty_j}{(L_{\mathcal{Y}'}(i))^2}\right).
                        \end{equation*}
                        The result follows as 
                        $$
                            L_{\mathcal{Y}'}(i)\ge \sum_{\substack{j=1\\j\neq \tau}}^{i}y_j,
                        $$
                        and we compensate for the assumption $y_\tau\ge y'$.
                    \end{proof}

                    In the following lemma we relax the conditioning, and condition on the values of $Z_i$ only up to $\tau$, instead of up to $t$. Recall $t_0=t/\log\log{t}$.

                    \begin{lemma}\label{lem:change}
                        Let $1 \le \tau \le t$ and $y_1,\ldots,y_\tau$ and $y'$ positive integers satisfying $|\sum_{j=1}^{s}y_j-\Ex{L(\tau-1)}|\le t \log^{2/3}t$ for every $1\le s \le \tau-1$. 
                        Denote the events $\{Z_1=y_1,\ldots,Z_\tau=y_\tau\}$ and $\{Z_1=y_1,\ldots,Z_\tau=y'\}$ with $\mathcal{W}$ and $\mathcal{W}'$ respectively. Then for any positive integer $2t_0\le i\le t$ with $i\neq \tau$ and any fixed integer $1\le k\le t_0$ we have
                        $$\left|\PPo{d_{i}(t)=k \mid \mathcal{W}}-\PPo{d_{i}(t)=k \mid \mathcal{W}'}\right|\le 		 k |y_\tau-y'|O\left(\frac{t}{t_0^2 \log {t}}\right). $$
                    \end{lemma}

                    \begin{proof}
                        For a vector $\xvec\in  \lceil t \log\log{t}\rceil^{t-\tau}$ let $\mathcal{W}(\xvec)$ and $\mathcal{W}'(\xvec)$ denote the events $\{Z_1=y_1,\ldots,Z_\tau=y_\tau,Z_{\tau+1}=x_1,\ldots,Z_t=x_{t-\tau}\}$ and $\{Z_1=y_1,\ldots,Z_\tau=y',Z_{\tau+1}=x_1,\ldots,Z_t=x_{t-\tau}\}$.
                        Note that due to the independence of the $Z_i$ for any $\xvec\in [ t \log\log{t}]^{t-\tau}$ we have $\PPo{\mathcal{W}(\xvec)\mid \mathcal{W}}=\PPo{\mathcal{W}'(\xvec)\mid \mathcal{W}'}$.
                        Since $k\le t_0< i \le t$ and $\tau\neq i$ by Lemma~\ref{lem:condchange} we have
                        \begin{align*}
                            &\left|\PPo{d_{i}(t)=k \mid \mathcal{W}}-\PPo{d_{i}(t)=k \mid \mathcal{W}'}\right|\\
                            &=\left|\sum_{\xvec\in\lceil t\log\log{t}\rceil^{t-\tau}}\PPo{\mathcal{W}(\xvec)\mid \mathcal{W}}\left(\PPo{d_{i}(t)=k \mid \mathcal{W}(\xvec)}-\PPo{d_{i}(t)=k \mid \mathcal{W}'(\xvec)}\right)\right|\\ 
                            &\le \sum_{\xvec\in\lceil t\log\log{t}\rceil^{t-\tau}}\PPo{\mathcal{W}(\xvec)\mid \mathcal{W}} \frac{3k |y_\tau-y'|}{\sum_{j=1}^{\min\{i,\tau-1 \}}y_j+\sum_{j=1}^{i-\tau} x_j}\left(1+\frac{y'+y_\tau+\sum_{j=1}^{t-\tau} x_j}{\sum_{j=1}^{\min\{i,\tau-1 \}}y_j+\sum_{j=1}^{i-\tau} x_j}\right).
                        \end{align*}
                        First assume $\tau\le t_0$. Then 
                        \begin{align*}
                            &\sum_{\xvec\in\lceil t\log\log{t}\rceil^{t-\tau}}\PPo{\mathcal{W}(\xvec)\mid \mathcal{W}} \frac{3k |y_\tau-y'|}{\sum_{j=1}^{\min\{i,\tau-1 \}}y_j+\sum_{j=1}^{i-\tau}x_j}\left(1+\frac{y'+y_\tau+\sum_{j=1}^{t-\tau} x_j}{\sum_{j=1}^{\min\{i,\tau-1 \}}y_j+\sum_{j=1}^{i-\tau}x_j}\right)\\
                            &\le \sum_{\xvec\in\lceil t\log\log{t}\rceil^{t-\tau}}\PPo{\mathcal{W}(\xvec)\mid \mathcal{W}} \frac{3k |y_\tau-y'|}{\sum_{j=1}^{i-\tau}x_j} \left(1+\frac{y'+y_\tau+\sum_{j=1}^{i-\tau} x_j + \sum_{j=i-\tau+1}^{t-\tau}x_j}{\sum_{j=1}^{i-\tau}x_j}\right)\\
                            & = 3k |y_\tau-y'|\sum_{\xvec\in\lceil t\log\log{t}\rceil^{t-\tau}}\PPo{\mathcal{W}(\xvec)\mid \mathcal{W}} \left( \frac{2}{\sum_{j=1}^{i-\tau}x_j}+ \frac{y'+y_\tau+ \sum_{j=i-\tau+1}^{t-\tau}x_j}{(\sum_{j=1}^{i-\tau}x_j)^2}\right)\\
                            &\le 3k |y_\tau-y'| \Ex{\frac{2}{\sum_{j=\tau+1}^{i}Z_j}+ \frac{y'+y_\tau+ \sum_{j=i+1}^{t}Z_j}{(\sum_{j=\tau+1}^{i}Z_j)^2}}\\
                            &\le \frac{3k |y_\tau-y'|}{(1+o(1))\beta'' t_0 \log{t}}\left(2+O\left(\frac{t\log t}{t_0 \log t}\right)\right)=k |y_\tau-y'| O\left(\frac{t}{t_0^2 \log {t}}\right),
                        \end{align*}
                        where in the last inequality we used  $i-\tau\ge t_0$ in order to apply Lemma~\ref{lem:Zmoments} and \eqref{eq:firstmomZ}. Further, note that the indices of the $Z_j$ in the numerator differ from the indices of the $Z_j$ in the denominator implying that they are independent.
                        
                        On the other hand if $\tau>t_0$ we have, due to our condition $|\sum_{j=1}^{\min\{i,\tau-1 \}}y_j-\Ex{L(\tau-1)}|\le t \log^{2/3}t$ and by \eqref{eq:edgconc} of Lemma \ref{treotto}, that 
                        $\sum_{j=1}^{\min\{i,\tau-1 \}}y_j \ge (1+o(1))\beta''t_0 \log{t}$.
                        Thus,
                        \begin{align*}
                            &\sum_{\xvec\in\lceil t\log\log{t}\rceil^{t-\tau}}\PPo{\mathcal{W}(\xvec)\mid \mathcal{W}} \frac{3k |y_\tau-y'|}{\sum_{j=1}^{\min\{i,\tau-1 \}}y_j+\sum_{j=1}^{i-\tau}x_j}\left(1+\frac{y'+y_\tau+\sum_{j=1}^{t-\tau} x_j}{\sum_{j=1}^{\min\{i,\tau-1 \}}y_j+\sum_{j=1}^{i-\tau}x_j}\right)\\
                            &\le \sum_{\xvec\in\lceil t\log\log{t}\rceil^{t-\tau}}\PPo{\mathcal{W}(\xvec)\mid \mathcal{W}} \frac{3k |y_\tau-y'|}{\sum_{j=1}^{\min\{i,\tau-1 \}}y_j}\left(1+\frac{y'+y_\tau+\sum_{j=1}^{t-\tau} x_j}{\sum_{j=1}^{\min\{i,\tau-1 \}}y_j}\right)\\
                            &\le \frac{3k |y_\tau-y'|}{\sum_{\substack{j=1}}^{\min\{i,\tau-1 \}}y_j}\left(1+O\left(\frac{t\log t}{t_0 \log t}\right)\right)= k |y_\tau-y'| O\left(\frac{t}{t_0^2 \log {t}}\right).
                        \end{align*}
                    \end{proof}

                    Recall that
                    $$
                        \numdeg_{k}(t)=\sum_{i=0}^t 1_{d_{i}(t)=k}.
                    $$
                    Instead of considering $\numdeg_{k}(t)$ we will consider a closely related random variable 
                    $$
                        \numdeg_k'(t)=\sum_{i=2t_0}^t 1_{d_{i}(t)=k},
                    $$
                    Clearly $\numdeg_k(t)=\numdeg_k'(t)+o(t)$.

                    We will use a martingale argument to show that $R_k'(t)$ is concentrated around its expectation, namely we inspect the martingale
                    $A_k(\tau) = G(\tau \wedge H)$ where
                    $G(\tau)= \Ex{\numdeg_k'(t)\mid Z_1,\ldots, Z_{\tau}}$ and $H = \min \{ s \ge 0 \colon |\bar{Z}(s)| > t\log^{2/3}t \}$, with $\bar{Z}(s) =\sum_{j=1}^{s}\left(Z_j-\Ex{Z}\right)$.
                    In the following lemma we prove an upper bound for the conditional variance of this martingale.

                    \begin{lemma}\label{lem:variance}
                        For every $1\le k \le t_0$ and for every $1\le \tau \le t$, we have
                        $$
                            |A_k(\tau)-A_k(\tau-1)|=O(t \log^{-1/2}t)
                        $$
                        and
                        $$
                            \Var{A_k(\tau)\mid Z_1,\ldots,Z_{\tau-1}}=O(t \log^{-1/2}t)
                        $$
                        almost surely.
                    \end{lemma}

                    \begin{proof}
                        Fix $\tau$. In order to arrive to the result we will show that for every $z_1,\ldots,z_{\tau-1}$ between $1$ and $t\log\log t+1$, 
                        conditional on $Z_1=z_1,\ldots,Z_{\tau-1}=z_{\tau-1}$, we have $|A_k(\tau)-A_k(\tau-1)|=O(t \log^{-1/2}t)$ and
                        $\Var{A_k(\tau)\mid Z_1=z_1,\ldots,Z_{\tau-1}=z_{\tau-1}}=O(t \log^{-1/2}t)$.
                        Again let $M(\tau) = \bar{Z}(\tau \wedge H)$ and define $\mathcal{A}$ as the maximal set, such that for every $\omega \in \mathcal{A}$ we have $(Z_1,\ldots, Z_{\tau-1})(\omega)=(z_1,\ldots,z_{\tau-1})$.
                        For $\omega \in \mathcal{A}$ satisfying $|M(\tau-1)(\omega)|> t\log^{2/3}{t}$ we have $A_k(\tau)(\omega)=A_k(\tau-1)(\omega)$ and $\Var{A_k(\tau) \mid Z_1,\ldots,Z_{\tau-1}}(\omega)=0$.
                        Hence, in the following we will consider the set $\mathcal{B}$ consisting of all $\omega \in \mathcal{A}$
                        such that
                        $|M(\tau-1)(\omega)|\le t\log^{2/3}{t}$. Therefore, $\omega\in \mathcal{C}_{\tau-1}$ and by Lemma~\ref{treotto} for any $2t_0\le i <\tau$,
                        $$
                            \sum_{j=1}^{i}z_{j}\ge (1+o(1))\beta'' i \log{t}-t\log\log{t}-1\ge (1+o(1))\beta'' t_0 \log{t}.
                        $$
	
                        Since $A_k(\tau)$ is a martingale we have, for each $\omega \in \mathcal{B}$, $\Ex{A_k(\tau) \mid Z_1,\ldots,Z_{\tau-1}}(\omega)=A_k(\tau-1)(\omega)$ and thus
                        \begin{align}
                            &\Var  {A_k(\tau) \mid  Z_1,\ldots,Z_{\tau-1}}(\omega)=
                            \Ex{ \left(A_k(\tau)-A_k(\tau-1)\right)^2 \mid Z_1,\ldots,Z_{\tau-1}}(\omega)\nonumber\\
                            &\quad \le \max \left[|A_k(\tau)(\omega)-A_k(\tau-1)(\omega)| \right]\cdot \Ex{|A_k(\tau)- A_k(\tau-1)| \mid Z_1,\ldots,Z_{\tau-1}}(\omega). \label{eq:var}
                        \end{align}
                        Now,
                        \begin{align}
                            & |A_k(\tau)(\omega)-A_k(\tau-1)(\omega)|=
                            \left|\Ex{\numdeg'_k(t)\mid Z_1,\ldots,Z_\tau}(\omega)- \Ex{\numdeg'_k(t) \mid Z_1,\ldots,Z_{\tau-1}}(\omega)\right|\nonumber\\
                            &\le \sum_{i=2t_0}^t \left|\Ex{1_{d_{i}(t)=k}\mid Z_1,\ldots,Z_\tau}(\omega)- \Ex{1_{d_{i}(t)=k} \mid Z_1,\ldots,Z_{\tau-1}}(\omega)\right|\nonumber\\
                            &\le \sum_{\substack{i=2t_0\\i\neq \tau}}^t \left|\Ex{1_{d_{i}(t)=k}\mid Z_1,\ldots,Z_{\tau}}(\omega)- \Ex{1_{d_{i}(t)=k}\mid Z_1,\ldots,Z_{\tau-1}}(\omega)\right|+1\nonumber\\
                            &=\sum_{\substack{i=2t_0\\i\neq \tau}}^t \left|\PPo{d_{i}(t)=k\mid Z_1,\ldots,Z_{\tau}}(\omega)- \PPo{d_{i}(t)=k\mid Z_1,\ldots,Z_{\tau-1}}(\omega)\right|+1.\label{eq:Adiff}
                        \end{align}
	
                        Let $\mathbb{P}_{Z_\tau=\ell}$ be the conditional probability given the event $\{Z_{\tau}=\ell\}$. 
                        Then for a fixed $2t_0 \le i \le t$ such that $i\neq \tau$ we have
                        \begin{align*}
                            &\left|\PPo{d_{i}(t)=k\mid Z_1,\ldots,Z_{\tau}}(\omega)- \PPo{d_{i}(t)=k\mid Z_1,\ldots,Z_{\tau-1}}(\omega)\right|\\
                            &\leq \sum_{\ell:1\le \ell < t\log\log{t}+1}\PPo{Z_\tau=\ell}\left|\PPo{d_{i}(t)=k\mid Z_1,\ldots,Z_{\tau}}(\omega)- \PPol{d_{i}(t)=k\mid Z_1,\ldots,Z_{\tau-1}}(\omega)\right|.
                        \end{align*}
                        Then by Lemma~\ref{lem:change}
                        \begin{align*}
                            & \sum_{\ell:1\le \ell < t\log\log{t}+1}\PPo{Z_\tau=\ell}\left|\PPo{d_{i}(t)=k\mid Z_1,\ldots,Z_{\tau}}(\omega)- \PPol{d_{i}(t)=k\mid Z_1,\ldots,Z_{\tau-1}}(\omega)\right|\\
                            & \le \sum_{\ell:1\le \ell < t\log\log{t}+1}\PPo{Z_\tau=\ell} k |Z_\tau-\ell| O\left(\frac{t}{ t_0^2 \log {t}}\right)\\
                            &\le \sum_{\ell:1\le \ell < t\log\log{t}+1}\PPo{Z_\tau=\ell} k (Z_\tau+\ell)O\left(\frac{t}{ t_0^2 \log {t}}\right)\\
                            &= O\left(\frac{t}{t_0^2}\left(\frac{Z_\tau}{\log{t}}+1\right)\right),
                        \end{align*}
                        where we used \eqref{eq:firstmomZ}.

                        Returning to $\eqref{eq:Adiff}$ we deduce 
                        $$
                            |A_k(\tau)(\omega)-A_k(\tau-1)(\omega)|\le O\left(\frac{t^2}{t_0^2}\left(\frac{Z_\tau}{\log{t}}+1\right)\right).
                        $$
                        Replacing $Z_{\tau}$ by its maximum gives us for every $\omega \in \mathcal{B}$
                        $$
                            |A_k(\tau)(\omega)-A_k(\tau-1)(\omega)|=O(t \log^{-1/2}t).
                        $$
                        On the other hand replacing $Z_{\tau}$ by its maximum and its expected value respectively together with \eqref{eq:var} implies for every $\omega \in \mathcal{B}$ that
                        $$
                            \Var  {A_k(\tau) \mid  Z_1,\ldots,Z_{\tau-1}}(\omega)=O\left(\frac{ t^5 \log\log{t}}{t_0^4 \log{t}}\right)=O(t \log^{-1/2}t).
                        $$
                    \end{proof}

                    Using the previous lemma we can show that the martingale $A_k(\tau)$ is concentrated around its expectation.

                    \begin{lemma}\label{lem:marconc1}
                        For any fixed integer $1\le k \le t_0$ we have  that under $\mathcal{C}$
                        the event $\{\Ex{\numdeg_k(t)\mid Z_1,\ldots, Z_t}=\Ex{\numdeg_k(t)}+o(t)\}$ holds whp.
                    \end{lemma}

                    \begin{proof}
                        By Theorem~\ref{thm:marconc} and Lemma~\ref{lem:variance} we have 
                        \begin{align}
                            \PPo{|A_k(t)-\Ex{A_k(t)}|\ge t\log^{-1/8}t}&\le 2\exp\left(-\frac{t^2 \log^{-1/4}t}{O(t^2 \log^{-1/2}t)+t \log^{-1/8}t O(t\log^{-1/2}t)} \right)\nonumber\\
                            &\le 2\exp\left(-\Omega(\log^{1/4}t)\right)=o(1). \label{eq:Akconc}
                        \end{align}
                        When $\mathcal{C}$ holds, which occurs whp by Lemma \ref{lem:edgeconc}, we have $A_k(t)=\numdeg'_{k}(t)$ and since $A_{k}(t),\numdeg'_k(t)\le t$ we also have $\Ex{A_k(t)}=\Ex{\numdeg'_k(t)}+o(t)$. 
                        Together with \eqref{eq:Akconc} we deduce that whp $\numdeg'_k(t)=\Ex{R_k'(t)}+o(t)$.
                        The result follows as $\numdeg_k(t)=\numdeg'_k(t)+o(t)$.
                    \end{proof}

                    Now we only have to consider what happens once the value of $Z_1,\ldots,Z_t$ have been established. This resembles the proof showing that the degree sequence in the classical preferential attachment model is concentrated.

                    \begin{lemma}\label{lem:marconc2}
                        For $y_1,y_2,\ldots,y_t$, let $\mathcal{Y}$ denote the event $\{Z_1=y_1,\ldots,Z_t=y_t\}$. If $\mathcal{Y}\subseteq \mathcal{C}$ we have 
                        $$
                            \PPo{|\numdeg_k(t)-\Ex{\numdeg_k(t)\mid \mathcal{Y}}|\ge \sqrt{t} \log{t}\mid \mathcal{Y}}=o(1/t).
                        $$
                    \end{lemma}
                    
                    \begin{proof}
                        We will break up every step into rounds, where in each round exactly one edge is inserted. Therefore step $\tau$ consists of $y_\tau$ rounds. In round $(j,\ell)$ we insert the $\ell$-th edge in the $j$-th step. Denote by $(j,\ell)-1$ the round preceding the $(j,\ell)$-th round, that is
                        $$
                            (j,\ell)-1=
                            \left\{
                            \begin{array} {ll}
                                (j,\ell-1), &  \mbox{if } \ell>1,\\
                                (j-1,y_{j-1}), & \mbox{if }  \ell=1 \mbox{ and } j>1,\\
                                (0,0), & \mbox{if } (j,\ell)=(1,1).
                            \end{array}
                            \right.
                        $$

                        Let $\mathcal{F}_{(j,\ell)}$ be the filtration formed by adding the edges one at a time, with $\mathcal{F}_{(0,0)}$ the trivial sigma algebra. Fix $1\le j \le t$ and $1\le \ell \le y_j$ and run a second copy of the truncated model, such that it matches the original run, for every round before round $(j,\ell)$ and afterwards runs independently. Denote by $d_{i}(\tau)$ the degree of vertex $v_i$ after $\tau$ steps and by $d_{i}'(\tau)$ the degree of $v_i$ after $\tau$ steps in the second copy.
                        Clearly,
                        $$
                            \Ex{\numdeg_k(t)\mid \mathcal{F}_{(j,\ell)}}=\sum_{i=1}^t \PPo{d_{i}(t)=k \mid \mathcal{F}_{(j,\ell)}},
                        $$
                        and since the vertex degrees are equidistributed and due to the two copies being independent after the first $(j,\ell)-1$ rounds we have
                        $$
                            \Ex{\numdeg_k(t) \mid \mathcal{F}_{(j,\ell)-1}}=\sum_{i=1}^t\PPo{d'_{i}(t)=k\mid \mathcal{F}_{(j,\ell)-1}}=\sum_{i=1}^t \PPo{d'_{i}(t)=k\mid \mathcal{F}_{(j,\ell)}}.
                        $$
                        Therefore
                        $$
                            \Ex{\numdeg_k(t)\mid \mathcal{F}_{(j,\ell)}}-\Ex{\numdeg_k(t) \mid \mathcal{F}_{(j,\ell)-1}}=\sum_{i=1}^t \left( \PPo{d_{i}(t)=k \mid \mathcal{F}_{(j,\ell)}}-\PPo{d'_{i}(t)=k\mid \mathcal{F}_{(j,\ell)}} \right).
                        $$

                        Note that conditional on $\mathcal{F}_{(j,\ell)}$ the degree of a vertex after round $(j,\ell)$ depends on $\mathcal{F}_{(j,\ell)}$, only through  the degree of the vertex at the end of round $(j,\ell)-1$ and the degree of the vertex until round $(j,\ell)$. Denote by $d_i((j,\ell))$ and $d_i'((j,\ell))$ the degree of vertex $v_i$ at the end of round $(j,\ell)$ in the first and second copy respectively. Then 
                        \begin{align*}
                            &\Ex{\numdeg_k(t)\mid \mathcal{F}_{(j,\ell)}}-\Ex{\numdeg_k(t) \mid \mathcal{F}_{(j,\ell)-1}}\\
                            &=\sum_{i=1}^t \mathbb{E}\biggl[\PPo{d_{i}(t)=k \mid d_{i}((j,\ell)-1),d_{i}((j,\ell))}-\PPo{d'_{i}(t)=k\mid d'_{i}((j,\ell)-1),d'_{i}((j,\ell))}\biggr| \mathcal{F}_{(j,\ell)}\biggr].
                        \end{align*}

                        Note that $d_{i}((j,\ell)-1)=d_{i}'((j,\ell)-1)$, 
                        thus 
                        $$
                            \PPo{d_{i}(t)=k \mid d_{i}((j,\ell)-1),d_{i}((j,\ell))}=\PPo{d'_{i}(t)=k\mid d'_{i}((j,\ell)-1),d'_{i}((j,\ell))}
                        $$
                        whenever $d_{i}((j,\ell))=d'_{i}((j,\ell))$. 
                        Therefore    
                        \begin{equation}
                            |\Ex{\numdeg_k(t)\mid \mathcal{F}_{(j,\ell)}}-\Ex{\numdeg_k(t) \mid \mathcal{F}_{(j,\ell)-1}}|\le \Ex{\sum_{i=1}^t 1_{d_{i}((j,\ell))\neq d'_{i}((j,\ell))}}\le 2,
                        \end{equation}
                        as the degree of exactly 2 vertices is changed in each round. This allows us to apply the Azuma--Hoeffding inequality. For $(j,\ell) = (t, y_t)$ we obtain
                        $$
                            \PPo{|\numdeg_k(t)-\Ex{\numdeg_k(t)\mid \mathcal{Y}}|\ge \sqrt{t}\log t \mid \mathcal{Y}}\le \exp\left(-\frac{t\log^2{t}}{O(t\log t)}\right)=o(1/t)
                        $$
                        where we used that since $\mathcal{Y}\subseteq \mathcal{C}$ there are $O(t\log{t})$ rounds during the first $t$ steps.
                    \end{proof}

                \subsubsection{Proof of Theorem~\ref{thm:main2}}\label{th:ma2:pr}

                    Note that Lemma \ref{trecinque} allows us
                    to consider suitably truncated random variables for the initial degrees.
                    Fix $k\ge 1$ then by Lemmas~\ref{lem:marconc1} and \ref{lem:marconc2} we have conditional on $\mathcal{C}$ whp $\numdeg_{k}(t)=\Ex{\numdeg_{k}(t)}+o(t)$. The result follows as by Lemma~\ref{lem:edgeconc} $\mathcal{C}$ holds whp and Lemma~\ref{lem:expR} implies $\Ex{\numdeg_k(t)}=(t+1)b_k+o(t)$ and $\ratdeg_k(t)=\numdeg_{k}(t)/(t+1)$.

                    Note that if $k:=k(t)\to \infty$ as $t\to \infty$ Lemma~\ref{lem:expR} implies that $\Ex{\numdeg_k(t)}=o(t)$. The result follows from Markov's inequality and $\ratdeg_k(t)=\numdeg_{k}(t)/(t+1)$.

                \subsubsection*{Acknowledgments}

                    T. Makai was supported by ERC Grant Agreement 772606-PTRCSP.

                    F.\ Polito and L.\ Sacerdote have been partially supported by the MIUR-PRIN 2022 project ``Non-Markovian dynamics and non-local equations'', no.\ 202277N5H9 and by INdAM/GNAMPA. 

                    L.\ Sacerdote has been partially supported by
                    the Spoke 1 ``FutureHPC \&
                    BigData'' of ICSC - Centro Nazionale di Ricerca in
                    High-Performance-Computing, Big Data and Quantum Computing, funded by
                    European Union - NextGenerationEU.

    \bibliography{pref_ref}

\begin{thebibliography}{10}

\bibitem{MR2091634}
Albert-L\'{a}szl\'{o} Barab\'{a}si and R\'{e}ka Albert.
\newblock Emergence of scaling in random networks.
\newblock {\em Science}, 286(5439):509--512, 1999.

\bibitem{MR3161480}
Noam Berger, Christian Borgs, Jennifer~T. Chayes, and Amin Saberi.
\newblock Asymptotic behavior and distributional limits of preferential
  attachment graphs.
\newblock {\em Ann. Probab.}, 42(1):1--40, 2014.

\bibitem{bhamidi2007universal}
Shankar Bhamidi.
\newblock Universal techniques to analyze preferential attachment trees: Global
  and local analysis.
\newblock {\em Available from
  \texttt{www.academia.edu/download/30790314/10.1.1.120.5134.pdf}}, 2007.

\bibitem{MR1824277}
B\'{e}la Bollob\'{a}s, Oliver Riordan, Joel Spencer, and G\'{a}bor Tusn\'{a}dy.
\newblock The degree sequence of a scale-free random graph process.
\newblock {\em Random Structures Algorithms}, 18(3):279--290, 2001.

\bibitem{MR2283885}
Fan Chung and Linyuan Lu.
\newblock Concentration inequalities and martingale inequalities: a survey.
\newblock {\em Internet Math.}, 3(1):79--127, 2006.

\bibitem{MR1966545}
Colin Cooper and Alan Frieze.
\newblock A general model of web graphs.
\newblock {\em Random Structures Algorithms}, 22(3):311--335, 2003.

\bibitem{MR4142215}
Umberto De~Ambroggio, Federico Polito, and Laura Sacerdote.
\newblock On dynamic random graphs with degree homogenization via
  anti-preferential attachment probabilities.
\newblock {\em Phys. D}, 414:132689, 12, 2020.

\bibitem{MR2480915}
Maria Deijfen, Henri van~den Esker, Remco van~der Hofstad, and Gerard
  Hooghiemstra.
\newblock A preferential attachment model with random initial degrees.
\newblock {\em Ark. Mat.}, 47(1):41--72, 2009.

\bibitem{MR2511283}
Steffen Dereich and Peter M\"{o}rters.
\newblock Random networks with sublinear preferential attachment: degree
  evolutions.
\newblock {\em Electron. J. Probab.}, 14:no. 43, 1222--1267, 2009.

\bibitem{MR3912097}
Alan Frieze, Xavier P\'{e}rez-Gim\'{e}nez, Pawe\l Pra\l~at, and Benjamin
  Reiniger.
\newblock Perfect matchings and {H}amiltonian cycles in the preferential
  attachment model.
\newblock {\em Random Structures Algorithms}, 54(2):258--288, 2019.

\bibitem{MR4312838}
Svante Janson and Lutz Warnke.
\newblock Preferential attachment without vertex growth: emergence of the giant
  component.
\newblock {\em Ann. Appl. Probab.}, 31(4):1523--1547, 2021.

\bibitem{MR4522354}
Joost Jorritsma and J\'{u}lia Komj\'{a}thy.
\newblock Distance evolutions in growing preferential attachment graphs.
\newblock {\em Ann. Appl. Probab.}, 32(6):4356--4397, 2022.

\bibitem{MR3997484}
Tomasz \L~uczak, Abram Magner, and Wojciech Szpankowski.
\newblock Asymmetry and structural information in preferential attachment
  graphs.
\newblock {\em Random Structures Algorithms}, 55(3):696--718, 2019.

\bibitem{MR4193887}
Bas Lodewijks and Marcel Ortgiese.
\newblock A phase transition for preferential attachment models with additive
  fitness.
\newblock {\em Electron. J. Probab.}, 25:Paper No. 146, 54, 2020.

\bibitem{MR4195181}
Yury Malyshkin.
\newblock Sublinear preferential attachment combined with a growing number of
  choices.
\newblock {\em Electron. Commun. Probab.}, 25:Paper No. 87, 12, 2020.

\bibitem{MR3776186}
Angelica Pachon, Laura Sacerdote, and Shuyi Yang.
\newblock Scale-free behavior of networks with the copresence of preferential
  and uniform attachment rules.
\newblock {\em Phys. D}, 371:1--12, 2018.

\bibitem{MR3668381}
Erol Pek\"{o}z, Adrian R\"{o}llin, and Nathan Ross.
\newblock Joint degree distributions of preferential attachment random graphs.
\newblock {\em Adv. in Appl. Probab.}, 49(2):368--387, 2017.

\bibitem{MR4269210}
Delphin S\'{e}nizergues.
\newblock Geometry of weighted recursive and affine preferential attachment
  trees.
\newblock {\em Electron. J. Probab.}, 26:Paper No. 80, 56, 2021.

\bibitem{MR3617364}
Remco van~der Hofstad.
\newblock {\em Random graphs and complex networks. {V}ol. 1}.
\newblock Cambridge Series in Statistical and Probabilistic Mathematics, [43].
  Cambridge University Press, Cambridge, 2017.

\bibitem{MR4467842}
Tiandong Wang and Panpan Zhang.
\newblock Directed hybrid random networks mixing preferential attachment with
  uniform attachment mechanisms.
\newblock {\em Ann. Inst. Statist. Math.}, 74(5):957--986, 2022.

\end{thebibliography}
    \bibliographystyle{plain}

\end{document}